\documentclass[11 pt,reqno]{amsart}
\usepackage{amsmath,amsthm,amsfonts,amssymb,mathrsfs,bm,graphicx,stmaryrd,dsfont}
\usepackage[usenames,dvipsnames]{color}
\usepackage[colorlinks=true,linkcolor=blue]{hyperref}
\usepackage[letterpaper,hmargin=1.0in,vmargin=1.0in]{geometry}
\parskip 	\smallskipamount

\newtheorem{theorem}{Theorem}[section]
\newtheorem{lemma}[theorem]{Lemma}
\newtheorem{corollary}[theorem]{Corollary}
\newtheorem{proposition}[theorem]{Proposition}

\newtheorem{definition}[theorem]{Definition}

\def\P{\mathbb{P}}
\def\Z{\mathbb{Z}}
\def\R{\mathbb{R}}

\def\E{\mathbb{E}}
\def\G{\Gamma}

\newcommand{\ce}{\mathcal{E}}
\newcommand{\cg}{\mathcal{G}}
\newcommand{\e}{\varepsilon}
\begin{document}

\title[Area Trapping Polymers]{The competition of roughness and curvature in area-constrained polymer models}

\author{Riddhipratim Basu}
\address{Riddhipratim Basu, Department of Mathematics, Stanford University, Stanford, CA, USA}
\email{rbasu@stanford.edu}
\author{Shirshendu Ganguly}
\address{Shirshendu Ganguly, Department of Statistics, UC Berkeley, Berkeley, CA, USA}
\email{sganguly@berkeley.edu}
\author{Alan Hammond}
\address{Alan Hammond, Departments of Mathematics and Statistics, UC Berkeley, Berkeley, CA, USA}
\email{alanmh@berkeley.edu}
\date{}
\maketitle
\vspace{-.1in}
\begin{abstract}
The competition between local Brownian roughness and global parabolic curvature experienced in many random interface models reflects an important aspect of the KPZ universality class. It may be summarised by an exponent triple $(1/2,1/3,2/3)$ representing local interface fluctuation, local roughness (or inward deviation) and convex hull facet length. The three effects arise, for example, in droplets in planar Ising models \cite{AH1,AH2,AH3,Alexander}. In this article, we offer a new perspective on this phenomenon. We consider  directed last passage percolation model in the plane, a paradigmatic example in the KPZ universality class, and constrain   the maximizing path under the additional requirement of enclosing an atypically large area. The interface suffers a constraint of parabolic curvature as before, but now its local structure is the KPZ fixed point polymer's rather than Brownian.  The local interface fluctuation exponent is thus two-thirds rather than one-half.  We prove that the facet lengths of the constrained path's convex hull are governed by an exponent of $3/4$, and inward deviation by an exponent of $1/2$. That is, the exponent triple is now $(2/3,1/2,3/4)$ in place of  $(1/2,1/3,2/3)$. This phenomenon appears to be shared among various  isoperimetrically extremal circuits in local randomness. Indeed, we formulate a conjecture to this effect concerning such circuits in supercritical percolation, whose Wulff-like first-order behaviour was recently established by Biskup, Louidor, Procaccia and Rosenthal in \cite{BLPR}.
\end{abstract}

\begin{figure}[h] 
\centering
\begin{tabular}{cc}
\includegraphics[scale=.15]{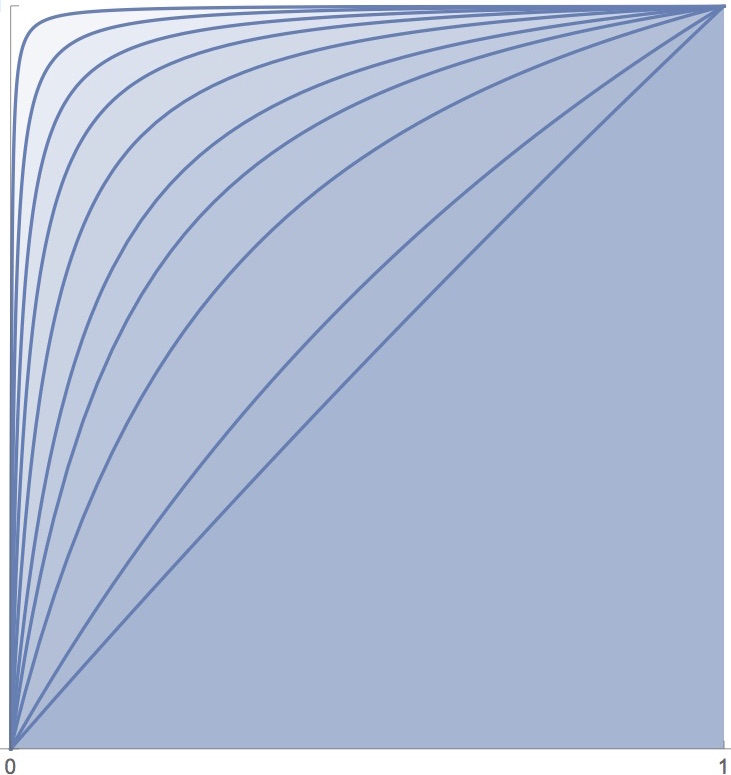} &\quad\quad\quad\quad\quad\quad\includegraphics[width=0.25\textwidth]{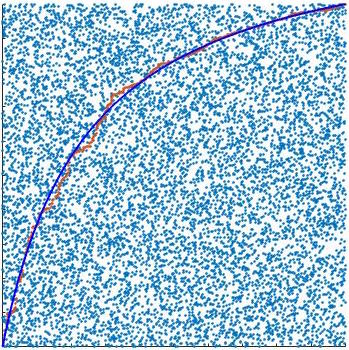} \\
(a) & \quad\quad\quad\quad\quad\quad(b)
\end{tabular}
\caption{ (a) Limiting curves for the constrained geodesics for various trapped area values. (b) A typical realization for the area trapping polymer model.} 
\label{fig0}
\end{figure}
\tableofcontents
\section{Introduction and main results}

The geometric properties of random interfaces are a vast arena of study in rigorous statistical mechanics.
Two important classes of interface models are {\em phase separation} models that idealize the boundary between a droplet of one substance suspended in another, and {\em last passage percolation} models, where a directed path in independent local randomness maximizes a random weight determined by the environment.

One of the best known mathematical examples of phase separation is the two dimensional supercritical Ising model in a large box with negative boundary condition. Conditioned to have an atypically high number of positive spins in the box, the vertices in the plus phase tend to form a droplet surrounded by a sea of negative spins. The random phase boundary of such a droplet has been the object of intense study.  
Wulff proposed that the profile of such constrained circuits would macroscopically resemble
a dilation of an isoperimetrically optimal curve. This was established rigorously in \cite{DKS} and \cite{IS}.
Via the FK representation of the Ising model, 
 one observes a similar situation in the setting of subcritical two dimensional percolation, where the analogous object is the boundary of the cluster containing the origin after this cluster is conditioned to be atypically large.  The Wulff shape captures the macroscopic profile of the circuit in such models, but what of fluctuations? Several definitions may be considered that seek to capture fluctuation behaviour on the part of the circuit, including  the deviation in the Hausdorff metric of the convex hull of the circuit from an appropriately scaled Wulff shape.
Alternative definitions, more local in nature, serve better to capture the transition in circuit geometry from local Brownian randomness to a smoother profile dictated by the constraints of global curvature.
In fact, a pair of definitions is natural, one to capture the longitudinal distance at which the transition takes place, and the second to treat the orthogonal inward deviation of the interface at this transition scale. 
To specify the characteristic longitudinal distance, we may note that the convex hull of the circuit is a polygonal path that is composed of planar line segments or {\em facets}; we may treat the typical or maximum facet length as a barometer of the transition from the shorter scale of local randomness to the longer scale of curvature. Latitudinally, we may note that any point in the circuit has a {\em local roughness}, given by its distance from the convex hull boundary. The typical or maximum local roughness along the circuit is a latitudinal counterpart to facet length. 
 Alexander \cite{Alexander}, and Hammond \cite{AH1,AH2,AH3} analysed such conditioned circuit models and determined that when the area contained in the circuit is of order $n^2$, so that the circuit has diameter of order $n$, facet length and local roughness scale as $n^{2/3}$ and $n^{1/3}$. A similar situation is witnessed when a parabola $x \to t^{-1}x^2$ is subtracted from a two-sided Brownian motion $B:\R \to \R$. When $t > 0$ is large, facets of the motion's convex hull have length $\Theta(t^{2/3})$ and inward deviation $\Theta(t^{1/3})$.
This phenomenon is expected to be universal when the local structure of the interface is Brownian and other examples include 
a Brownian bridge pinned at $-T$ and $T$ and conditioned to remain above the semi-circle of radius $T$ centred at the origin which was analysed by Ferrari and Spohn \cite{FS}. 

The second class of random geometric paths we have mentioned are last passage percolation models. These models form part of a huge, Kardar-Parisi-Zhang, class of statistical mechanical models in which a path through randomness is selected to be extremal for a natural weight determined by that randomness. The maximizing paths are often called {\em polymers}. 
The fluctuation behaviour of a length $n$ polymer with given endpoints may be gauged either in terms of the scale of deviation of its weight from the mean value, or by the scale of deviation of say the polymer's midpoint from the planar line segment that interpolates the polymer's endpoints. The two deviations are given by scales $n^{1/3}$ and $n^{2/3}$. 
These were first proved for planar Poissonian
directed last passage percolation, in the seminal work of Baik, Deift and Johansson \cite{BDJ99}. Since then, this and other integrable models in the same KPZ universality class have been extensively analysed and detailed information about this model, both geometric and algebraic, has been obtained \cite{J00, LM01, LMS02}. 

It is of much interest to study models that combine phase separation and path minimization (or maximization) in random environment. Consider an environment with local independent randomness, and the circuit through that randomness whose weight determined by that randomness is extremal among those circuits that trap a given area. Studying the random geometry of such a circuit is a problem of extremal isoperimetry. Taking the randomness to be supercritical percolation, Itai Benjamini conjectured the first-order, Wulff-like behaviour of the boundary of the set in the supercritical percolation cluster attaining the so called anchored expansion constant. This was proved by~\cite{BLPR} by showing that the curve in the limit solves a natural isoperimetric variational problem; this has recently been extended to higher dimensions by Gold \cite{G16}. Given the role of facet length and local roughness in capturing the local random to global curvature transition in circuit geometry, the natural next problem is to understand the scaling exponents of such objects, which is the pursuit we undertake in this paper.  
Our choice of model preserves the qualitative features of the problem of interface fluctuation in an isoperimetrically  extremal droplet in supercritical percolation and  at the same time ensures that key  algebraic aspects of KPZ theory can be harnessed to yield sharp fluctuation estimates.

The particular setting we consider in this paper is planar Poissonian 
directed last passage percolation. 
To impose the required properties, we study this model under a quadratic curvature constraint: that is, we force the best path to move away from the straight line and have a quadratic curvature on the average. This is done naturally by considering the longest upright path in a Poissonian environment joining $(0,0)$ and $(n,n)$ which has the additional \textit{area trap} property that the area under the curve is at least $(\frac{1}{2}+\alpha)n^2$ for some  $\alpha\in (0,\frac{1}{2})$. Note that from the discussion above it follows that the unconstrained longest path stays close to the diagonal and hence encloses area $(\frac{1}{2}+o(1))n^2$. 
Postponing precise statements until later (see Section \ref{scale-exp}), our main results quantify the competition between the global parabolic nature  and local behaviour guided by KPZ relations of the contour.
They show that the maximum facet length of the contour's least concave majorant scales as  $n^{3/4+o(1)}$, while the inward deviation from the concave majorant (local roughness) 
scales as $n^{1/2+o(1)}$.

We also establish a law of large numbers for the length of the optimal path. The proof  proceeds by setting up a variational problem, as is natural in such contexts (see \cite{DZ2}, \cite{BLPR}). Perhaps  surprisingly, however, the problem turns out to have a very explicit solution. The geometric information about the limit shape of the constrained polymer is also used as input in some of the arguments about fluctuations. 
The question of fluctuation in the context of  isoperimetrically extremal circuits in supercritical percolation can be formulated as a first passage percolation analogue of our setting of last passage percolation. Although the first passage percolation model is not exactly solvable, it, too, is believed to be in the KPZ universality class, and the geodesics there are believed to have the same $n^{2/3}$ scaling of transversal fluctuation as in our case 
(see \cite{FPPsurvey} and the references therein). 
Thus our results suggest that a certain universality is being witnessed and, according to this belief, we formulate a conjecture concerning percolation. 
We elaborate on this and a number of other interesting questions in Section \ref{oq}.

Finally, we discuss briefly the key inputs used in our paper and how it contrasts with the other examples of fluctuation results already mentioned.
The results in \cite{AH1,AH2,AH3} crucially used the refined understanding and geometric estimates for percolation clusters while a study of area trapping planar Brownian loop \cite{HP} used well known estimates for Brownian motion.  Exact expressions involving Brownian motion conditioned to stay over a parabola was also the key ingredient in the proofs in \cite{FS}. 
However, in our setting, even though the unconstrained model has integrable properties, the lack of general geometric understanding as one deviates slightly from integrable models causes a big challenge for us.  
Nonetheless, using certain known facts about the unconstrained model and their robust variants established recently in \cite{BSS14} as blackbox estimates, we rigorously establish local roughness exponents for our model which we henceforth call the \emph{Area Trapping Polymer} model (see next section for precise definitions). 
Thus, this work is an example  where one can use inputs from the integrable literature to make geometric conclusions about settings which are beyond the exactly solvable world (see also \cite{BSS14}).
We believe a general program of developing such geometric arguments would lead to more robust proofs which could work for settings that are non-integrable but are variants of solvable models. 

\subsection{Model Definitions}\label{def1012}
We now recall the planar Poissonian directed last passage percolation model. 

Let $\Pi$ be a homogeneous rate one Poisson Point Process (PPP) on the plane. A partial order on $\R^2$ is given by $(x_1,y_1)\preceq (x_2,y_2)$ if $x_1\leq x_2$ and $y_1\leq y_2$. 
For $u\preceq v$, a directed path $\gamma$ from $u$ to $v$ is a piecewise linear path that joins points $u=\gamma_{0}\preceq \gamma_1 \preceq \cdots \preceq \gamma_{k}= v$ where each $\gamma_i$  for $i\in [k-1]$ (throughout the article we will adopt the standard notation $[n]=\{1,2,\ldots,n\}$)
is a point of $\Pi$. 
Define the length of $\gamma$, denoted $|\gamma|$, to be the number of $\Pi$-points on $\gamma$. 
\begin{definition}\label{geodesic1} 
 Define the last passage time from $u$ to $v$, denoted by $L(u,v)$, to be the maximum  of $|\gamma|$ as $\gamma$ varies over all directed paths from $u$ to $v$. There may be several maximizing paths between $u$ and $v$, and throughout the paper we will refer to the top most path (it is easy to see that the top most path is well defined here) among those, as the \emph{geodesic} between $u$ and $v$ and denote it by  $\gamma(u,v)$. 
 \end{definition}

We will often call $\gamma(u,v)$ as the polymer between $u$ and $v$ and $|\gamma(u,v)|$ as the polymer length/weight respectively.
Next we introduce a constraint in this classical model. 

\subsubsection{Area Trapped by a path}\label{def10}
Consider a path $\gamma$ between the origin $(0,0)$ and a point $(x,y)$ in the positive quadrant. The area trapped by the path $\gamma$, denoted by $A(\gamma)$, is defined to be the area of the closed polygon determined by the $x$-axis, the vertical line segment joining $(x,0)$ to $(x,y)$ together with the line segments of the path $\gamma$. Let $\gamma_n$ denote the geodesic  between $(0,0)$ and $(n,n)$. Well-known facts about Poissonian LPP readily imply that $A(\gamma_{n})= (\frac{1}{2}+o(1))n^2$ asymptotically almost surely. We constrain the model and consider maximizing paths subject to trapping a much larger area. To this end fix $\alpha\in (0,\frac{1}{2})$, and let 
\begin{equation}
\label{e:cop}
L_{\alpha}(n):=\max\left\{|\gamma|: \gamma~\text{path from}~(0,0)~\text{to}~(n,n)~\text{and}~ A(\gamma)\geq \big(\tfrac{1}{2}+\alpha\big)n^2\right\}.
\end{equation}
It is easily seen that, among the paths that attain this maximum, there is almost surely exactly one that traps the least area. This path will be called $\Gamma_{\alpha,n}$. We will write   $\Gamma_{n}$ provided that the context clarifies the value of $\alpha$ in question. We shall call $\Gamma_n$ the constrained (or $\alpha$-constrained) geodesic.
In analogy with the phase separation example in percolation mentioned the introduction, it might seem more natural to consider the family of  down right paths joining $(n,0)$ and $(0,n)$ since all of these curves enclose the origin. However because of the obvious underlying symmetry, and to take advantage of standard notational conventions, throughout the sequel  we will consider the contour joining the origin to the point $(n,n)$.

Our main objects of interest are two quantities that measure local regularity of the constrained geodesic $\Gamma_n$. The following definitions are illustrated by Figure \ref{def}.

\begin{definition}
\label{d:mlrmfl}
Let ${\rm conv}(\Gamma_{\alpha, n})$ denote the convex hull of the polygon determined by the path $\Gamma_{\alpha, n}$ and the coordinate axes. Let $\Gamma^{*}_{\alpha, n}$ denote the closure of the polygonal part of the boundary of ${\rm conv}(\Gamma_n)$ between $(0,0)$ and $(n,n)$ above the $x$-axis. Thus $\Gamma^{*}_{\alpha, n}$ is the least concave majorant of $\Gamma_{\alpha,n}$ and is an union of finitely many line segments. These  segments will be called {\rm facets}.
Define {\rm{maximum facet length}} of $\Gamma_n$, denoted $\mathrm{MFL}(\Gamma_{\alpha, n})$ to be the maximum Euclidean length of the facets. For $x\in \Gamma_{\alpha, n}$, let ${\rm d}(x,\Gamma^{*}_{\alpha, n})$ denote the distance from $x$ to $\Gamma^*_{\alpha, n}$. This is a natural notion of  the \emph{local roughness at $x$}. 
Define the {\rm{maximum local roughness}} of $\Gamma_n$, 
denoted $\mathrm{MLR}(\Gamma_n)$ by
$$\mathrm{MLR}(\Gamma_{\alpha, n}):= \sup\{x\in \Gamma_{\alpha, n}: {\rm d}(x,\Gamma^{*}_{\alpha, n})\}.$$
\end{definition} 

In the following subsections we present our main results. We first state our results regarding the fluctuation exponents of $\Gamma_{\alpha, n}$.
\subsection{Scaling exponents for geodesic geometry}\label{scale-exp}
The sense in which we capture the exponents is stronger and easier to state for the lower bounds, and so we begin with them. 

\begin{theorem}
\label{t:mfllb}
Fix $\alpha \in (0,\frac{1}{2})$ and $\varepsilon>0$. Then   there exists $c=c(\alpha,\e)>0$ such that for all large enough $n,$
$$\P(\mathrm{MFL}(\Gamma_n) \geq n^{3/4-\varepsilon})\ge 1-e^{-n^c}.$$ 
\end{theorem}

As we have mentioned, the scaling exponent for transversal fluctuation of point-to-point geodesic in unconstrained Poissonian last passage percolation is known to be $2/3$ 
This fact, put together with the above theorem yields the following lower bound on maximum local roughness.

\begin{theorem}
\label{t:mlrlb}
Fix $\alpha \in (0,\frac{1}{2})$ and $\varepsilon>0$. Then   there exists $c=c(\alpha, \e)>0$ such that for all large enough $n,$
$$\P(\mathrm{MLR}(\Gamma_n) \geq n^{1/2-\varepsilon})\ge 1-e^{-n^c}.$$ 
\end{theorem}

Regarding the matching upper bound, 
we prove that, with high probability, there exists a dense set of $\alpha \in (0,2^{-1})$ for which the maximum length of the facets away from the boundary is bounded above by $n^{3/4+o(1)}$. To make this precise, fix $\delta \in (0,\pi/4)$, and consider a facet in $\Gamma_{\alpha, n}$ with endpoints $A$ and $B$
recorded in clockwise order. Setting $O = (n,0)$,  let $\theta_{A}$ denote the acute angle that $OA$ makes with the $y$-axis and $\theta_{B}$, the acute angle  that $OB$ makes with the $x$-axis; see Figure~\ref{def}. 

\begin{definition}
\label{interior1}
The facet $AB$ is called $\delta$-interior if $\min (\theta_{A}, \theta_{B})\geq \delta$.
\end{definition} 

Note that the union of the $\delta$-interior facets forms a polygonal path.
(The union could be empty, but we will infer later from Theorem~\ref{lln2}
that this event has an exponentially small probability.)
\footnote{Note that at this point, it is not a priori clear if the set of $\delta-$interior facets is non-empty. However, this will be a consequence of Theorem \ref{lln2}, stated later, which implies for any $\delta$, the length of all the facets 
will be less than $O(\delta n)$ with exponentially small failure probability.} Let $A_0$ and $B_0$ denote the extremities of this union path, and let $\Gamma_{\delta, \alpha, n}$ denote the subpath of  $\Gamma_{\alpha,n}$ between $A_0$ and $B_0$. Define the maximum $\delta$-interior facet length of $\Gamma_{\alpha, n}$, denoted by $\mathrm{MFL}(\Gamma_{\delta, \alpha,n})$, to be the maximum length of the $\delta$-interior facets.  Define the $\delta$-interior maximum local roughness, denoted by $\mathrm{MLR}(\Gamma_{\delta, \alpha, n})$ by altering  Definition \ref{d:mlrmfl} so that now the supremum is taken over all $x \in \Gamma_{\delta, \alpha, n}$. 
\begin{definition}\label{good10}
Fix $\varepsilon>0$. We say $\alpha\in (0,\frac{1}{2})$ is  $(n,\varepsilon,\delta)$-good if $\mathrm{MFL}(\Gamma_{\delta, \alpha, n}) \leq n^{3/4+\varepsilon}$. 
\end{definition}
Here is our upper bound concerning the maximum length of facets.

\begin{theorem}
\label{t:mflub}
Fix $\varepsilon>0$, $\delta\in (0,\pi/4)$ and an interval $[\alpha_1, \alpha_2] \subset (0,\frac{1}{2})$. Let $I_{n, \alpha_1,\alpha_2}$ denote the set of  $(n,\varepsilon, \delta)$-good $\alpha \in [\alpha_1, \alpha_2]$. Then there exists $c=c(\varepsilon,\delta,\alpha_1,\alpha_2)>0$ such that the  probability that  $I_{n, \alpha_1,\alpha_2} \not= \emptyset$ is at least $1-e^{-n^c}$ for all large enough $n.$
\end{theorem} 

Our final principal result concerning exponents asserts that, for $\alpha\in I_{n,\alpha_1,\alpha_2}$, it is highly likely that  $\mathrm{MLR}(\Gamma_{\delta, \alpha, n})\geq n^{1/2+\varepsilon}$. Thus we obtain,  with high probability,  an $n^{1/2+\varepsilon}$ upper bound for interior maximum local roughness for a dense, though possibly $n$-dependent and random, set of~$\alpha$. 

\begin{theorem}
\label{t:mlrub}
In the setting of Theorem \ref{t:mflub}, there exists $c>0$, such that the event that there exists $\alpha\in I_{n, \alpha_1,\alpha_2}$ such that 
$\mathrm{MLR}(\Gamma_{\delta, \alpha, n})> n^{1/2+\varepsilon},$ occurs with probability at most  $e^{-n^c} $ for all large enough $n$.
\end{theorem}

\begin{figure}[h] 
\centering
\begin{tabular}{cc}
\includegraphics[width=0.38\textwidth]{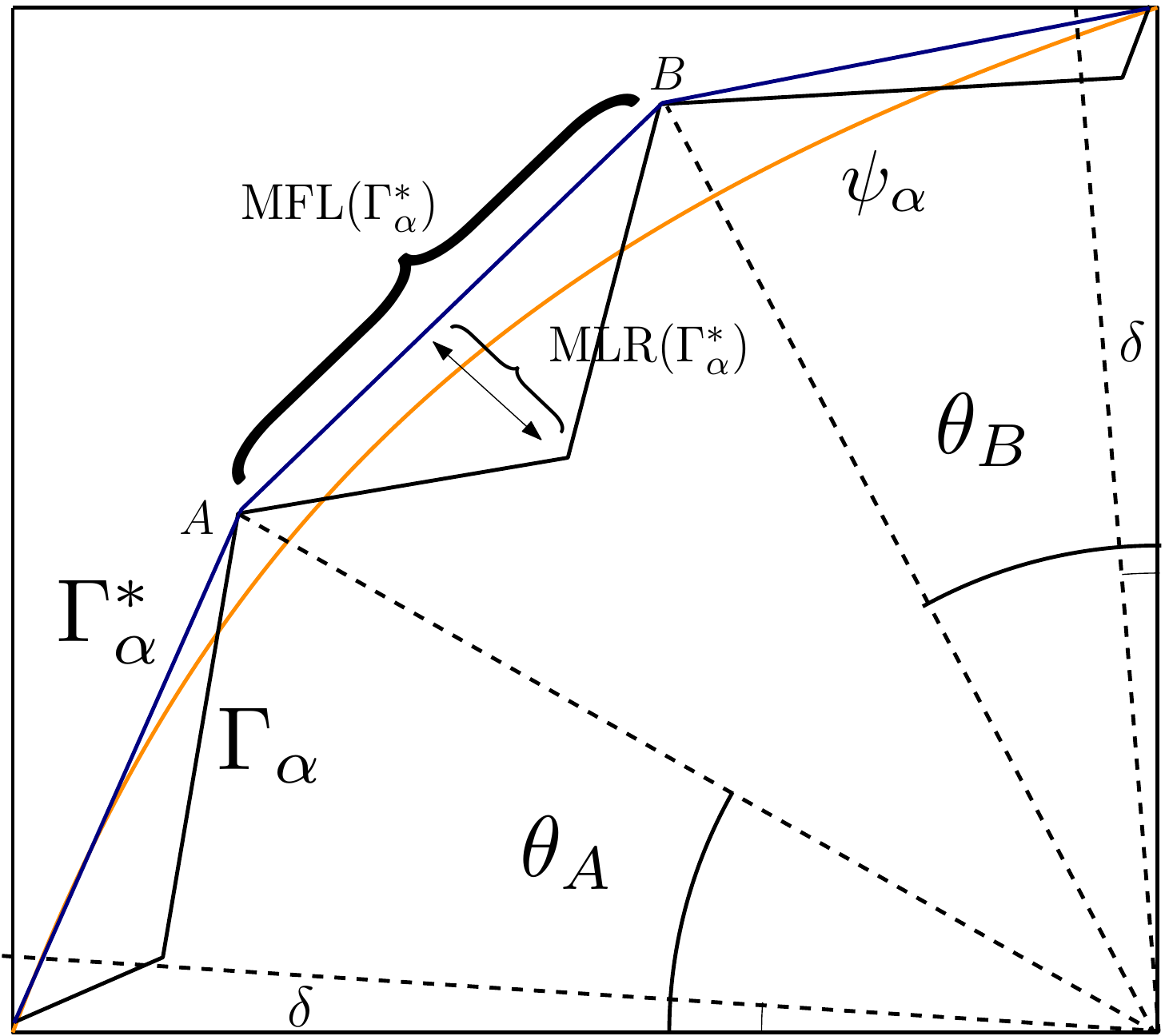} &\quad\includegraphics[width=0.4\textwidth]{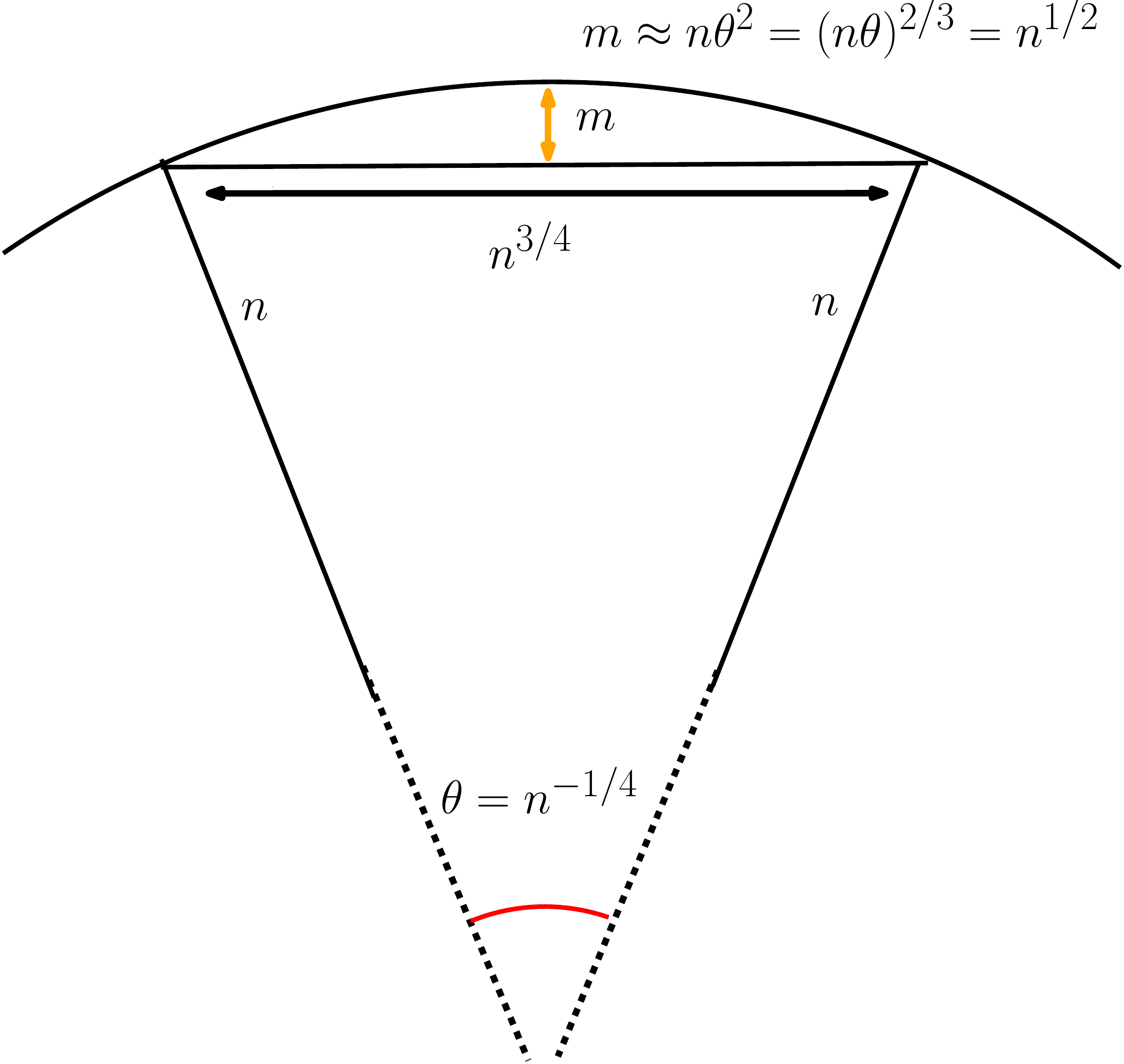} \\
(a) & (b)
\end{tabular}
\caption{ (a) The three different curves correspond to $\psi_{\alpha,n},\Gamma_{\alpha,n}, \Gamma^*_{\alpha,n}$.  The  dashed lines indicate the $\delta-$ interior and hence the facets inside the sector bounded by them, are the $\delta-$interior facets.
(b)The scale at which the competition from two sources are equal to each other.}
\label{def}
\end{figure}

We now state our results about law of large numbers  for $L_{\alpha}(n)$ and the constrained geodesic. 
Although the route taken here of setting up an appropriate variational problem is by now classical \cite{DZ2}, we point out that it was not at all obvious that the solution can be explicitly described. Moreover some of the consequences of the results in the following section are used as geometric inputs for the fluctuation results stated before.  
\subsection{Law of Large Numbers}
For the unconstrained model, a straightforward subadditivity argument yields that $\frac{\E L_{n}}{n}$ converges to a limit. The evaluation of the limiting constant is classical: \cite{VerKer77, LogShep77} showed that the limit equals $2$ by using Young tableaux combinatorics and the RSK correspondence (see also \cite{AD95}). However, for the constrained model, the subadditive structure is lost and it is not clear a priori that a law of large numbers for $L_{\alpha}(n)$ exists. Our first result here is to establish the law of large numbers for the area trapping polymer model; we are also able to evaluate the limiting constant implicitly as a function of $\alpha$. 

Given $\alpha \in (0,1/2)$, let $c_{\alpha}$ be given implicitly by the following equation:
\begin{equation}\label{implicit1}
\frac{1+c_{\alpha}}{c_{\alpha}}\left(1-\frac{\log(1+c_{\alpha})}{c_{\alpha}}\right)=\frac{1}{2}+\alpha.
\end{equation}
One can see that the function $f(c)=\frac{1+c}{c}[1-\frac{\log(1+c)}{c}]$ is strictly increasing \footnote{$f'(c)=\frac{(2+c)\log(1+c)-2c}{c^3}$ and $\frac{d}{dc}[(2+c)\log(1+c)-2c]=\log(1+c)-\frac{c}{1+c} >0$.} in $c$, and converges to $1/2$ and $1$ at $0$ and $\infty$ respectively. 
Let ${\mathrm{w}}_{\alpha}= \sqrt{1+c_{\alpha}}\frac{\log(1+c_{\alpha})}{c_{\alpha}}$.
\begin{theorem} 
\label{t:lln} 
For any $\alpha\in (0,1/2)$ 
$$\E L_{\alpha}(n)=2{\mathrm{w}}_{\alpha} n+o(n)$$
as $n\to \infty$.
\end{theorem}

Notice that ${\mathrm{w}}_{\alpha}\to 1$ as $\alpha\to 0$; hence the above theorem  
is consistent with the result in the classical unconstrained case. In the unconstrained model, one also has a law of large numbers for the geodesic, i.e., it is known that the geodesic is concentrated around the straight line joining $(0,0)$ and $(n,n)$. More precisely, under the rescaling that takes the $n\times n$ square to a unit square, the geodesic converges almost surely in Hausdorff distance to the diagonal of the unit square \cite{DZ2} (as mentioned before, the precise order of the transversal fluctuations are known to be $n^{2/3}$ as a consequence of integrability). Even though we do not have any exactly solvable structure in the constrained model, we can establish a similar law of large numbers here too asserting that the constrained geodesic concentrates around a deterministic curve. Moreover, we can identify the limiting curve in a fairly explicit manner. 

Let $\psi_{\alpha}:[0,1]\to [0,1]$ be defined as 
$\psi_{\alpha}(x)=\frac{(1+c_{\alpha})x}{1+c_{\alpha}x}$ where $c_{\alpha}$ is as above. Also let $\psi_{\alpha,n}: [0,n]\to [0,n]$ be the $n$ blow up of $\psi_{\alpha}$ i.e. $\psi_{\alpha,n}(x)=n \psi_{\alpha}(x/n)$. We denote by ${\rm dist}(\cdot,\cdot)$ the Hausdorff distance. 
The following theorem is our law of large numbers for the constrained geodesic.

\begin{theorem}\label{lln2} For any $\alpha \in (0,1/2)$ and $\Delta>0$, there exists $c=c(\alpha,\Delta)$ such that, for all large enough $n$, it is  with probability at least $1-e^{-cn}$ that $${\rm dist}(\Gamma_{\alpha,n},\psi_{\alpha,n})\le \Delta n.$$
\end{theorem}

In particular, this theorem states that to first order, the constrained geodesics behave like a given smooth curve. This result will be useful to us while studying the scaling exponents for local roughness of the constrained geodesic, in particular, when we see to  rule out long and flat facets (see Theorem \ref{flat0}).

\subsection{Open questions and future directions}\label{oq}

Below we list  below several interesting questions for future research:
\begin{enumerate}
\item The upper bound Theorem \ref{t:mflub} is weaker than the lower bound  Theorem \ref{t:mfllb}. Strengthening the former is a natural open problem. 
\item By definition, $A(\Gamma_{\alpha,n})\ge (\frac{1}{2}+\alpha)n^2$. The typical order of $A(\Gamma_{\alpha,n})- (\frac{1}{2}+\alpha)n^2$ remains unknown. 

\item For $\alpha \in (0,1/2)$,
what is the typical deviation of $\Gamma_{\alpha,n}$ from the curve $\psi_{\alpha,n}$? What is the order of fluctuations of $L_{\alpha}(n)$?
\item In a phase separation problem, \cite{AH1,AH2,AH3} determines the polylogarithmic corrections to both elements of the $(2/3,1/3)$ (facet length,local roughness) exponent pair. The powers of the logarithm are $(1/3,2/3)$. Finding such corrections for the \emph{Area Trapping Polymer} model would refine  the identification of the exponent pair $(3/4,1/2)$ made in this article and suggested by the first point.
The first two points will be discussed further in 
Section~\ref{areafluc}.
\end{enumerate}
\noindent
\textbf{Supercritical Percolation}: We end this section with a conjecture regarding fluctuation exponents in the context of supercritical percolation on the nearest neighbor graph on $\Z^2$. As already mentioned before, \cite{BLPR} settles a conjecture of Benjamini regarding the limit shape of  isoperimetrically extremal  sets. Formally, for supercritical percolation, with the origin $\mathbf{0}$ conditioned to be in the infinite cluster $\mathcal{C}_{\infty}({\bf{0}})$, the authors in \cite{BLPR} consider  the `anchored isoperimetric profile', i.e. for any $r>0$ they look at the  set $\mathcal{B}_{r}$ which solves the following isoperimetric problem:
$$\inf \left\{\frac{|\partial B|}{|B|}: {\bf{0}}\in B, B\subset \mathcal{C}_{\infty}({\bf{0}}) \text{ is connected}, |B|\le r \right\}$$ where $\partial B$ denotes the edge boundary of $B,$  restricted to $\mathcal{C}_{\infty}(\bf{0})$ (for more details see \cite{BLPR}).
The main result in \cite{BLPR} shows the convergence of the set $\mathcal{B}_r$  as $r\to \infty$ after suitable renormalization, to a deterministic Wulff crystal, in the Hausdorff sense.  To go beyond first order behaviour one has to understand local  geometric properties of the boundary of  $\mathcal{B}_r$. 
The  extremal circuit broadly has to satisfy the:
\begin{enumerate}
\item Volume condition
\item Extremal isoperimetry condition
\end{enumerate}
The heuristic now is that the former is a global constraint while the latter is only local. 
Namely, each local part of the boundary does not feel the volume  constraint $|\mathcal{B}_r|\le r,$ and thereby just tries to move through the `open' edges while trying to minimize the number of open edges that it cuts across, since these are precisely the edges that contribute to $|\partial \mathcal{B}_r|$.
 This brings us within the realm of first passage percolation  predicted to be in the KPZ universality class as well(see \cite{FPPsurvey} for more on first passage percolation).
 Thus Theorems, \ref{t:mfllb},  \ref{t:mlrlb},  \ref{t:mflub}  and \ref{t:mlrub}  can be thought of as rigorous   counterparts to the above discussion in the integrable last passage percolation setting and combined with the above discussion allows us to leads us to conjecture fluctuation exponents for the above model. Below we formulate a precise statement in a slightly simpler  setting: 
 
 Recall the standard definition of the dual graph of  the nearest neighbor lattice on $\Z^2$. 
Given a supercritical bond percolation environment on the edge set of $\Z^2$ with density $p>p_c(\Z^2)=1/2$, for every positive integer $n,$ consider the set $\mathcal{D}_{n^2},$ of dual  circuits (simple loop consisting of dual edges) enclosing a connected subset  of $\Z^2$ of size  $n^2$ and containing the origin such that the number of primal open edges in the percolation environment that the circuit cuts through is minimized. Note that such a circuit is not necessarily unique due to the discrete nature of the problem. 
Thus the above model is an exact analogue of the model considered in this paper, in the context of  first passage percolation in a bond percolation environment.
We now state precisely our conjecture:  

\noindent
\textbf{Conjecture}:
Consider  bond percolation on the nearest neighbor lattice on $\Z^2$ with any supercritical parameter value 
  $p>p_c(\Z^2)$.  The random variables 
$$\max_{D\in \mathcal{D}_{n^2}}\left|\frac{\log(\mathrm{MFL}(D))}{\log n}-3/4\right| \, \, \text{ and } \, \,  \max_{D\in \mathcal{D}_{n^2}}\left|\frac{\log(\mathrm{MLR}({D}))}{\log n}-1/2\right|$$ converge to zero  in probability   as $n$ grows to infinity
where for any $D\in \mathcal{D}_{n^2}$, the maximum facet length $\mathrm{MFL}(D)$ and maximum local roughness $\mathrm{MLR}(D)$ are defined in the same way as in this paper by considering the convex hull of the points in $D$.

We end with the remark that  the problem considered in \cite{BLPR}
corresponds to a similar first passage percolation problem where the environment has bounded dependence range and hence should have the same fluctuation behaviour as the simpler model
just described. 
 
\subsection{Organization of the rest of the article}
 In Section \ref{s:aux2} we collect some preliminary probabilistic results: some for the constrained model and a few from the unconstrained model. The results for the unconstrained model follow from the sharp moderate deviation estimates in \cite{LM01, LMS02} and the consequences established in \cite{BSS14}. In Section~\ref{s:var}, we set up and solve the variational problem for the law of large numbers and the following Section \ref{s:lln} is devoted to the proofs of Theorem \ref{t:lln} and Theorem \ref{lln2}: the law of large numbers for the length of constrained geodesic and the path itself.
 We next turn to the proofs of the results on the scaling exponents. In Section \ref{s:lb} we provide proofs of Theorems \ref{t:mfllb} and Theorem \ref{t:mlrlb}. The proofs of Theorem \ref{t:mflub} and Theorem \ref{t:mlrub} are completed in Section \ref{s:ub}. Proofs of some of the auxiliary results stated and used throughout the article are postponed to Section \ref{es}.

\subsection*{Acknowledgments}
The authors thank Marek Biskup, Craig Evans and Ofer Zeitouni for useful discussions. S.G.'s research is supported by a Miller Research Fellowship at UC Berkeley. A.H. is supported by NSF grant DMS-1512908. 

\section{Important probability estimates}
\label{s:aux2}
In this section we gather the  probabilistic inputs needed for our proof. First we shall record a few useful facts about the length and geometry of the constrained geodesics that will be used repeatedly later. The bulk of this section will then recall results about the unconstrained model. These results are all consequences of the exactly solvable nature of the Poissonian LPP model and can be derived starting with the basic integrable ingredients of the exactly solvable model obtained by Lowe, Merkl and Rolles \cite{LM01,LMS02}. Some of these consequences were established and used recently in \cite{BSS14} and we quote the relevant results here.

We start with some basic facts about the constrained model.

\subsection{Some Basic Results for the Area Trapping Polymer Model}
We first start with concentration for $L_{\alpha}(n)$. In absence of integrability we use the standard Poinc\'are inequality techniques and obtain a concentration at $n^{1/2+o(1)}$ scale. 

\begin{theorem}\label{concentration1} Fix any $\alpha \in (0,1/2)$. Then there exists a constant $C>0$ such that for all $t>0$ $$\mathbb{P}(|L_{\alpha}(n)-\E(L_{\alpha}(n))| \ge t) \le Ce^{-t^{2}/Cn \log^{2}(n)}.$$
\end{theorem} 
The proof is standard and is postponed until the end of the paper in Section \ref{es}. 

Our next result  rules out flat facets.
This is a consequence of Theorem \ref{lln2}, which asserts that it is extremely unlikely that any interior facets (which are those in the bulk) have  very shallow or steep gradient.
\begin{theorem}\label{flat0} For any small enough $\delta>0$ there exists $\gamma>0$ such that,   with probability $1-e^{-cn}$, all $\delta-$interior facets make an angle with the $x-$axis which lies in the interval $(\omega,\pi/2-\omega)$.
\end{theorem}

 \begin{proof} 
We start by recalling a  simple fact: 
 for two facets of $\Gamma^*_{\alpha,n}$ with starting points $u_1,u_3$ and ending points $u_2, u_4$ respectively where  $u_1 \preceq u_2 \preceq u_3 \preceq u_4,$ by convexity of $\Gamma^*_{\alpha,n}$, the angle made with the $x-$axis by the  facet $(u_1,u_2)$ is larger than the angle made by $(u_3,u_4)$. 


Consider any $\delta-$interior facet $(u_1,u_2),$ and let $u_{1}=(x,y)$.
Also let $L_1,L_2$ be the straight lines joining the origin to $u_1$ and $(n,n)$ to $u_2$ respectively. 
Let $v_1$ and $v_2$ be the points of intersection of $L_1$ and $L_2$ with $\psi_{\alpha,n}$. Theorem \ref{lln2} implies that, for any $\e > 0$, with  probability at least $1-e^{-cn},$  the bound  
$$\max (|u_1-v_1|,|u_2-v_2|)\le \e n$$
holds for all $\delta$-interior facets, simultaneously. 
Now, by the convexity of $\Gamma^{*}_{\alpha, n}$, the gradient of the facet $(u_1,u_2)$ is between the gradient of the lines joining $(0,0)$ to $u_1$ and $(n,n)$ to $u_2$.  Choosing $\e$ to be much smaller than $\delta,$ we see that the gradient of the facet $(u_1,u_2)$ plus an error of $O(\e)$ lies between the gradients of the lines joining  the origin and $v_1$ and $(n,n)$ and $v_2$ respectively. 
Thus we are done by choosing $\e$ to be small enough and using the strict convexity of $\psi_{\alpha,n}$; (see Figure \ref{def}(a) for illustration). 
\end{proof}

We will later need Theorem \ref{flat1}, a strengthening 
of the above result, which is uniform 
in $\alpha$.  
We next state useful results that concern the unconstrained, exactly solvable, model. All of these are corollaries of the following moderate deviation estimates.

\subsection{Moderate deviation estimates}
Recall the polymer length 
$L(\cdot,\cdot)$ from Definition \ref{geodesic1}.
Let $u=(u_1,u_2)\preceq v=(v_1,v_2)$ be such that $|u_1-v_1||u_2-v_2|=t$; i.e., $t$ is the area of the rectangle whose opposite corners are $u$ and $v$. Notice that,  by scale invariance of the Poisson point process, the distribution of $L(u,v)$ is a function merely of $t$.

\begin{theorem}[\cite{LM01,LMS02}]
\label{t:moddev}
Fix $\kappa>1$. Let $u$, $v$ as above be points such that the straight line joining $u$ and $v$ has gradient $m$, where 
$m\in ( \kappa^{-1} , \kappa)$. There exist positive constants $s_0$, $t_0$, $C$ and $c$ depending only on $\kappa$ such that, for all $t>t_0$ and $s>s_0$,
$$\P[|L(u,v)-2\sqrt{t}| \geq st^{1/6}]\leq Ce^{-cs^{3/2}}.$$
\end{theorem}

Observe that the above theorem implies that $|\E L(u,v) - 2\sqrt{t}|=O(t^{1/6})$; and hence similar tail bounds are true for the quantity $|L(u,v)-\E L(u,v)|$. When we make use of Theorem \ref{t:moddev}, it will often be in order to obtain tail bounds for the quantities $|L(u,v)-\E L(u,v)|$. Also the results of \cite{BDJ99} establish that $t^{-1/6} (L(u,v)-2\sqrt{t})$ converges weakly to the  GUE Tracy-Widom distribution (which is defined for example in \cite{BDJ99}). This law has negative mean, leading to the next bound.  

\begin{equation}
\label{e:mean}
\E L(u,v)\leq 2\sqrt{t}-Ct^{1/6}
\end{equation}
for some $C>0$.

\subsection{Transversal fluctuations of point-to-point geodesics}
Let $u=(u_1,u_2)\preceq v=(v_1,v_2)$. Let $y=p_{u,v}(x)$ denote the equation of the line segment joining $u$ and $v$. 
For any path $\gamma$ from $u$ to~$v$, define the maximum transversal fluctuation of the path $\gamma$ by 
$${\rm TF}(\gamma):= \sup_{x\in [u_1,v_1]} \{\sup |p_{u,v}(x)-y|: (x,y)\in \gamma\}.$$
For points $u=(u_1,u_2)$ and $v=(v_1,v_2)$, we  call $|v_1-u_1|$ the $x$-coordinate distance between $u$ and $v$ and similarly define the $y$-coordinate distance. 

Transversal fluctuations for paths between $(0,0)$ and $(n,n)$ were shown to be $n^{2/3+o(1)}$ with high probability in \cite{J00}. The following more precise estimate was established in \cite{BSS14}.

\begin{theorem}[Tails for Transversal Fluctuations, \cite{BSS14}]
\label{t:tftail}
Fix $\kappa>1$. Let $u\preceq v$ be points such that the $x$-coordinate distance between $u$ and $v$ is $r$ and the gradient of the line joining $u$ and~$v$ is $m$, where $m\in (\kappa^{-1}, \kappa)$. Then, for the geodesic $\Gamma$ from $u$ to $v$, 
$$\P[{\rm TF}(\Gamma) > sr^{2/3}]\leq Ce^{-cs}$$
for $\kappa$-dependent constants  $C,c$, and $r,s$ sufficiently large. 
\end{theorem}

The proof of Theorem \ref{t:tftail} 
from \cite{BSS14} in fact shows something more. Any path that has  transversal fluctuations much bigger than $r^{2/3}$ 
is not only unlikely to be a geodesic; it is typically much shorter than a geodesic. In particular, the proof of Theorem 11.1 in \cite{BSS14} implies the next result (which  can also be derived more straightforwardly).

\begin{theorem}[Paths with off-scale ${\rm TF}$ are uncompetitive]
\label{t:tftail1}
Fix $\kappa>1$ and $\e>0$. Let $u\preceq v$ be points such that the $x$-coordinate distance between $u$ and $v$ is $r$ and the gradient of the straight line joining $u$ and $v$ is $m$ where $m\in ( \kappa^{-1}, \kappa)$. 
Let $\ce=\ce_{u,v}$ denote the event that there exists a path $\gamma$ from $u$ to $v$ with ${\rm TF}(\gamma) \geq r^{2/3+\e}$ and $|\gamma|> \E[L(u,v)]-r^{1/3+\e/10}$ (we suppress the dependence of $\e$ in $\ce$ for brevity). Then there exists $c=c(\e)>0$ such that
$$\P(\ce) \leq e^{-r^{c}}$$ 
for $r$ sufficiently large. 
\end{theorem}

The stretched exponential nature of these estimates allow us to prove  uniform properties of the noise space which will be convenient for some of the arguments later. This point is a theme in this article: points to which we seek apply the above results may be special, so that estimates that are uniform over all points will be needed. 
For the next result, recall that our noise space is nothing other than a rate one Poisson point process in the box  $[0,n]^2$. 
For $u,v\in [0,n]^2$, let $\mathds{L}(u,v)$ denote  the line joining $u$ and $v$. Now, for any $u\preceq v$, the pair $(u,v)$ is called $\kappa$-steep if the gradient 
of $\mathds{L}(u,v)$ lies in the interval $[\kappa^{-1},\kappa]$. Lastly,
for $\tau>0$,  
\begin{equation}\label{notation540}
S(u,v)= S_{\tau}(u,v):=\left\{ \begin{array}{cc}
|u-v|^{1/3+\e/10} & \text{ if } |u-v|>n^{\tau},\\  
n^{2\tau/3} & \text{ otherwise }.
\end{array}\right. 
\end{equation}

\begin{corollary}
\label{uniform1}
Fix  $\e>0$, $\kappa>1$, and $\tau$ sufficiently small. Then there exists $c=c(\e,\tau)$ such that
$$\mathbb{P}\left(\bigcup_{u,v: (u,v) \text{is } \kappa-\text{steep}}
|L(u,v)-\E(L(u,v))|\ge S(u,v)\right) \le e^{-n^c}.$$
\end{corollary}
\begin{proof}
This follows by using Theorem \ref{t:moddev} and a standard coarse graining argument. Since we are not attempting to prove optimal bounds, we simply partition the box into squares of area one. Thus the corner points are the lattice points, i.e., elements of the set $\{0,1,\ldots,n\}\times\{0,1,\ldots,n\}$. The proof now follows by first proving the desired statement for lattice points $u,v$. This just follows by Theorem \ref{t:moddev} and a union bound over all pairs of lattice points $u,v$ (of which there are $n^2$). The proof is then completed by  showing that, since any box of area one is unlikely (with stretched exponentially small failure probability) to contain more than $n^{\e/100}$ points, for any $u$ and $v$, one can approximate $L(u,v)$ by $L(u_*,v_*)$, where $u_*,v_*$ are the nearest lattice points.
\end{proof}

We now specify a modification of the event $\ce_{u,v}$ from Theorem \ref{t:tftail1}.  Let $\ce^{*}_{u,v}$ denote the event $\ce_{u,v}$ where the permitted error term $r^{1/3+\e/10}$ is instead taken to be the  quantity~$S(u,v)$ defined a few moments ago. The next corollary is proved along the same lines as Theorem~\ref{t:tftail1}, in combination with the arguments in the proof of Corollary \ref{uniform1},
and we omit the proof. 
\begin{corollary}
\label{uniform0}
Fix  $\e>0$, $\kappa>1$, and $\tau$ small enough. Then there exists $c=c(\e,\tau)$ such that
$$\mathbb{P}\left(\bigcup_{u,v:|(u,v) \text{is } \kappa-\text{steep}}\ce^*_{u,v}\right) \le e^{-n^c}.$$
\end{corollary}

\subsection{Paths between points in an on scale parallelogram}
Next we control the deviation of lengths between pairs of points within certain parallelograms. We first need some definitions. Let  $U(r,m,h)$ denote a parallelogram with the  properties that:
\begin{enumerate}
\item[(a)] One pair of parallel sides are vertical.
\item[(b)] The other pair of parallel sides have gradient $m$.
\item[(c)] The (horizontal) distance between the vertical sides is $r$.
\item[(d)] The height of the vertical sides is $h$.
\end{enumerate}
Note that this defines a unique parallelogram up to translation. This will be enough for our purposes because the underlying noise field is translation invariant. 

\begin{theorem}[\cite{BSS14}]
\label{t:par1}
Let $m\in (\kappa^{-1},\kappa)$ and $h=W r^{2/3}$ for some $W>0$. Then there exist $(\kappa,W)$-dependent constants $C,c>0$,  such that, for all sufficiently large $r$ and $s$, 
$$ \P\left( \sup_{u\in A, v\in B} (L(u,v)- \E L(u,v)) \geq  sr^{1/3} \right) \leq Ce^{-cs},$$
and
$$ \P\left( \inf_{u\in A, v\in B} (L(u,v)- \E L(u,v)) \leq  -sr^{1/3} \right) \leq Ce^{-cs},$$
where $A$ and $B$ respectively denote the right third and the left third of $U(r,m,h)$.
\end{theorem}

The next theorem states that, even if the paths are restricted to stay inside the parallelogram, the fluctuation remains on scale. 
Let $L(u,v;U)$ denote the length of the longest path from $u$ to $v$ that does not exit $U$. 

\begin{theorem}[\cite{BSS14}]
\label{t:par2}
Under the assumptions of Theorem \ref{t:par1}, there exist $(\kappa,W)$-dependent constants $C,c>0$  such that, for all sufficiently large $r$ and $s$, 
$$ \P\left( \inf_{u\in A, v\in B} (L(u,v;U)- \E L(u,v)) \leq  -sr^{1/3} \right) \leq Ce^{-cs^{1/2}}$$
where $A$ and $B$ respectively denote the right third and the left third of $U(r,m,h)$.
\end{theorem}

\subsection{Paths constrained in a thin parallelogram}
The next set of results shows that, if a path is constrained to have much smaller than typical transversal fluctuation, then it must have much smaller length than a typical geodesic's.
\begin{theorem}[Paths with small ${\rm TF}$ are uncompetitive]
\label{t:constrained}
Fix $\e>0$. Consider the parallelogram $U=U(r,m,h)$, where $m\in (\kappa^{-1},\kappa)$ and $h=r^{2/3-\e}$. Let $u_0$ and $v_0$ denote the midpoints of the vertical sides $A_0$ and $B_0$ of $U$. Then, for some $c=c(\e)>0$,
$$\P\left(L(u_0,v_0; U) \geq \E L(u_0,v_0)- r^{1/3+\e/3}\right) \leq e^{-r^{c}}$$
for $r$ sufficiently large.
\end{theorem}

The proof of Theorem \ref{t:constrained} follows a strategy, by now well known, that has been used to show Gaussian fluctuation of paths constrained to stay in a thin rectangle in \cite{CD13} in the context of first passage percolation. In the context of LPP, the Gaussian fluctuation has recently been shown in \cite{DPM17}, and tail bounds are proved in a more general context in the preprint \cite{BH17}. We now  provide a sketch of Theorem~\ref{t:constrained}'s proof. 

\begin{proof}[Sketch of Proof]
Without loss of generality let us assume $m=1$. Also, we can replace the parallelogram $U$ by the rectangle $U',$ one of whose pairs of sides is parallel to the line segment $u_0v_0$, with the other pair having midpoints $u_0$ and $v_0$ 
and  length $r^{2/3-\e}$. Divide the rectangle $U'$ into $K$ equal parts (each of width $\frac{r}{K}$) by parallel line segments perpendicular to $u_0v_0$. Let $L_{i}$ denote the left side of the $i$-th rectangle and let $u_{i}$ be its midpoint. Now let $\gamma_{i}$ be the best 
path (i.e. with maximal value of $|\gamma|$) between $L_{i}$ and $L_{i+1}$ that stays within $U'$. It follows from Theorem \ref{t:par1} that it is extremely likely that $|\gamma_{i}|-L(u_i,u_{i+1})\ll (K^{-1}r)^{1/3}$. It follows that up to an error of order much smaller than $K^{2/3}r^{1/3}$ we can approximate $L(u_0,v_0; U')$ by $\sum_{i} L(u_i,u_{i+1})$. Now observe that $L(u_i,u_{i+1})$ are independent and identically distributed with mean $2r/K-C(K^{-1}r)^{1/3}$ (here we use~\eqref{e:mean}) and standard deviation of the order of $ (K^{-1}r)^{1/3}$. By standard concentration inequalities, one then shows that $\sum_{i} L(u_i,u_{i+1})$ concentrates around $2r-K^{2/3}r^{1/3}$ at scale $K^{1/6}r^{1/3}$. By choosing $K\gg r^{\e/2}$ properly one gets the result. 
\end{proof}

In the same vein, we have the following corollary.
\begin{corollary}
\label{uniform} 
Fix  $\e>0$, $\kappa>1$, and $\tau$ small enough. Then there exists $c=c(\e,\tau)$ such that, with probability at least $1-e^{-n^c}$,
$$
\displaystyle{\sup_{u,v:|u-v|\ge n^{\tau}, (u,v) \text{is } \kappa-\text{steep}} L(u,v,U)-\E(L(u,v)) \le  -|u-v|^{\frac{1}{3}+\e}}.
$$
\end{corollary}

\subsection{Estimates for One-Sided Geodesics}
We end this section by considering geodesics in a different constrained model. Baik and Rains \cite{br01,br} considered increasing 
paths from $(0,0)$ to $(n,n)$ that lie above the diagonal line joining the points. Recall that $L(u,v)$ denotes the length of the longest increasing path between points $u \preceq v$. Let $L^{\boxslash}(u,v)$ denote the length of the  longest path between $u$ and $v$ restricted to lie 
above the line joining $u$ and $v$ and, similarly to Definition \ref{geodesic1}, let $\gamma^{\boxslash}(u,v)$ denote the corresponding uppermost one-sided geodesic between $u$ and $v$. Moreover, let $L^{\boxslash}(n)$ denote $L^{\boxslash}(u,v)$ in the special case that $u=(0,0)$ and $v=(n,n)$. Baik and Rains \cite{br01,br} proved that $\E L^{\boxslash}(n)=2n+o(n)$ and that the fluctuations are again of order $n^{1/3}$ (although in this case the scaling limit is different; it is the GSE Tracy Widom distribution instead of the GUE Tracy Widom distribution). We shall need the corresponding moderate deviation estimates, consequences of  Theorem \ref{t:par2}. 

\begin{theorem}
\label{baikrains} 
Let $u,v$ be as in the hypothesis of Theorem \ref{t:moddev}. There exist positive constants $s_0, t_0, C$ and $c$ such that, for all $t>t_0$ and $s>s_0$,
$$\P[|L^{\boxslash}(u,v)-2\sqrt{t}| \geq st^{1/6}]\leq Ce^{-cs^{1/2}}.$$
\end{theorem}

\begin{proof}
The upper tail result follows from Theorem \ref{t:moddev} due to $L^{\boxslash}(u,v)\leq L(u,v)$. For the lower tail, we use Theorem \ref{t:par2} with the straight line joining $u$ and $v$ being the bottom side of the parallelogram and use the fact that $\E L(u,v)=2\sqrt{t}- \Theta(t^{1/6})$.
\end{proof}

We need a uniform version of this result. Recall $S(u,v)=S_{\tau}(u,v)$ from \eqref{notation540}.
\begin{corollary}
\label{uniform2}
Fix  $\e>0$, $\kappa>1$, and $\tau > 0$ small enough. Then there exists $c=c(\e,\tau)$ such that, with probability at least $1-e^{-n^c}$,
$$
\displaystyle{\sup_{u,v:(u,v) \text{is } \eta-\text{steep}} L^{\boxslash}(u,v)-\E(L(u,v)) \ge  -S(u,v)} .
$$
\end{corollary}

Proofs of Corollaries \ref{uniform0} \ref{uniform}, \ref{uniform2} follow from the corresponding theorems for fixed points, in the same way as  Corollary \ref{uniform1} follows from Theorem \ref{t:moddev}.  We omit the details. 

\section{A Variational Approach to the Constrained Geodesic}
\label{s:var}
We now move towards proving Theorem \ref{t:lln} and Theorem \ref{lln2}.
Deuschel and Zeitouni \cite{DZ1,DZ2} studied in detail the hydrodynamic limit of the unconstrained geodesic for inhomogeneous point processes. We follow their strategy broadly and study the limit of the constrained curve by means of appropriate variational problems. 
We first  explain the idea. Fix $\alpha\in (0,\frac{1}{2})$ for the rest of this section. First let us assume that a limiting continuous curve $\phi$ of the constrained geodesics exists after scaling. This curve $\phi:[0,1] \to [0,1]$ will be  continuous, non-decreasing and surjective.  By the area constraint, $\int_0^1\phi(s)ds \ge \tfrac{1}{2}+\alpha$. Now heuristically, since in practice there is always a little bit of area excess,  one can approximate $\phi$ by a piecewise affine function at a scale local enough that each piece is exempt from the area constraint; and we can then use the law of large numbers for the unconstrained geodesic (see Theorem \ref{t:moddev}) at every local scale and  sum over them. This argument suggests that the approximating path will have length $2J({\phi})n +o(n)$ where  
$$J({\phi})=\int_{0}^1\sqrt {\dot{\phi}(s)}ds.$$
It thus seems that we should maximize $J({\phi})$ among all curves $\phi$ satisfying the area constraint, and the maximum will be the constant appearing in the required law of large numbers. We now proceed to make this precise. 
Let $\mathcal{B}$ be the collection of all right-continuous non-decreasing functions from $[0,1]$ to $[0,1]$. 
Thus $\mathcal{B}$ is in bijection with the set of all sub-probability measures on $[0,1]$. Now, for any $\phi\in \mathcal{B},$ by the Lebesgue decomposition theorem we can write   
 \begin{equation}\label{decomposition23}
 \phi=\phi_{ac}+\phi_{s}
 \end{equation}
  in a unique way as a sum of a pair of sub-probability measures, with $\phi_{ac}$ being the absolutely continuous part and $\phi_s$ the  singular part (with respect to Lebesgue measure). 
Note that this implies $\phi_{ac}$ has a derivative 
$\dot \phi_{ac}$ almost everywhere \footnote{Throughout the article, for any function $f$ on $[0,1],$ which is differentiable almost everywhere, we will denote its derivative by $\dot f$. Recall, that this does not necessarily imply $\int_{0}^x \dot f(s)ds=f(x)$ for all $x$.}  
such that, for any $0\le x\le 1$, we have $$\int_{0}^x \dot \phi_{ac}(s)ds =\phi_{ac}(x)$$ while $\phi_s$ is almost surely flat and hence has derivative $0$ almost surely.  Thus $\dot \phi=\dot \phi_{ac}$ almost surely.
Also define 
\begin{equation}\label{variational1}
J(\phi)= \int_{0}^1\sqrt {\dot{\phi}(s)}{\rm d}s .
\end{equation} 
Let 
\begin{equation}\label{restr304}
\mathcal{B}_{\alpha}=\left\{\phi\in \mathcal{B}: \int_0^1\phi(s)ds \ge \tfrac{1}{2}+\alpha\right\}.
\end{equation}
We shall often omit the subscript $\alpha$. Finally let 
\begin{equation}\label{functional}
J_{\alpha}=\sup_{\phi \in \mathcal{B}_{\alpha}} J(\phi).
\end{equation}

The first step is to show existence and uniqueness of the solution. 

\begin{proposition}
\label{opt1} 
There exists a unique element $\psi=\psi_{\alpha} \in \mathcal{B}$ that attains the supremum  in \eqref{functional}.
\end{proposition}

Note that the function $\psi=\psi_{\alpha}$ in Proposition \ref{opt1} will be the same in the statement of Theorem~\ref{lln2}. The existence and uniqueness parts of Proposition \ref{opt1} have separate proofs. We first state the existence result.

\begin{lemma}
\label{l:exist}
There exists $\psi\in \mathcal{B}$ that achieves the supremum  in \eqref{functional}.
\end{lemma}

The proof of Lemma \ref{l:exist} is technical: it uses a compactness argument on the space of probability measures following a similar argument from \cite{DZ2}. We postpone the proof until Section \ref{es}. 

In the next part we show uniqueness. 

\begin{lemma}
\label{l:uni}
Suppose $\psi_1$ and $\psi_2$ are in $\mathcal{B}_{\alpha}$ and satisfy  $J(\psi_i)= J_{\alpha}$ for $i=1,2$. Then $\psi_1=\psi_2$. 
\end{lemma}

The proof of Lemma \ref{l:uni} is rather straightforward once we establish the following technical lemma that rules out the possibility that any of the optimizing functions have a nontrivial singular part. 

\begin{lemma}
\label{l:ac}
 Let $\psi \in \mathcal{B}_{\alpha}$ be such that $J(\psi)= J_{\alpha}$. Then $\psi$ corresponds to a probability measure which is absolutely continuous with respect to Lebesgue measure.
\end{lemma}

Here is the basic idea of the proof of this lemma. From the results of \cite{DZ2} it follows that without the area constraint the optimizing curve is the diagonal line, i.e., the preferred gradient for the graph of the function is one. If the singular part is non-trivial,  then the graph of the function $\psi$ may be expected to have a flat piece. Because the gradient of the graph has to be one on average due to boundary conditions, it follows that there must be a piece of the graph having gradient away from one. We shall show that one can modify the flat and the steep parts of the curve locally by pieces with more moderate gradient in a way that increases not only the value of the functional $J$; and  thereby optimality will be contradicted. Also, clearly in the case that $\psi$ does not have total mass one, one can add to it another function to make it (or, more precisely of course, to bring it  into correspondence with)  a probability measure, while also increasing the value of $J(\psi),$ and thus contradict maximality in this case also. Thus $\psi$ is also a probability measure. The details of the proof are postponed to Section \ref{es}.

We can now prove the uniqueness result using Lemma \ref{l:ac}. 

\begin{proof}[Proof of Lemma \ref{l:uni}]
We use the fact that the square-root  function is strictly concave. Given any $\psi_1$ and $\psi_2$ as in the statement of the lemma, we consider $\psi=\frac{\psi_1+\psi_2}{2}$.  Clearly, $\psi$ satisfies the area constraint and $\dot{(\frac{\psi_1+\psi_2}{2})}=\frac{\dot {\psi_1}+\dot{\psi_2}}{2}$. Thus using Jensen's inequality $J(\psi)$ is necessarily larger than $J(\psi_1)=J(\psi_2)$ unless $\dot \psi_1=\dot \psi _2$ almost surely. Since by the previous result the absolutely continuous parts contain mass one, $\psi_1(0)=\psi_2(0)=0$. Hence for all $0\le x\le1 $ 
$$\psi_1(x)=\int_0^{x}\dot \psi_1(s)ds=\int_0^{x}\dot \psi_2(s)ds=\psi_2(x).$$ Thus we are done. 
\end{proof}

It still remains to identify the optimizer $\psi$ in \eqref{functional}; in particular we need to show this is the same $\psi$ as defined in Theorem \ref{lln2}. To this end, we shall now record some properties of the unique optimizer $\psi$ that will be useful later. We show that the derivative is almost surely positive and decreasing. The proofs are easy and will be postponed until Section \ref{es}. 
\begin{lemma}\label{decreasing1} 
Let $\psi=\psi_{\alpha}$ denote the unique optimizer in Proposition \ref{opt1}. Then
\begin{enumerate}
\item[(i)] The derivative $\dot \psi$ is almost surely decreasing i.e. there is a set of full measure on which it is decreasing. 
\item[(ii)] $\dot \psi$ is almost surely positive.
\end{enumerate}
\end{lemma}

\subsection{Identifying the curve $\psi_{\alpha}$}
In this subsection we determine the curve $\psi_{\alpha}$ and show that it is the same as the curve in Theorem \ref{lln2}. We first start with the following proposition.

\begin{proposition}
\label{solution1} 
Given $\alpha \in (0, 1/2)$ there exists $c=c_\alpha>0,$ such that
$\psi_\alpha(x)=\frac{(1+c)x}{1+cx}$.
\end{proposition}

\begin{proof}The proof proceeds by showing that, given any $0<a<b<1$, $$\frac{1}{\sqrt{\dot \psi(x)}}=c_1x+c_2 \text{ a.s. on } [a,b].$$  
for some constants $c_1$ and $c_2$. Fix $0<a<b<1$ and let $h_*$ be a polynomial such that 
\begin{align}
\label{int20}
\int_a^b h_*(x){\rm{d}}x  &=0\\
\label{int21}
\int_a^b x h_*(x){\rm{d}}x &=0
\end{align}
Define $h_*(x)$ on $[0,1]$ by defining it to be $0$ outside $[a,b]$. 
Let $g=\dot \psi+\e h_*$. Now since $\dot \psi$ is monotonically decreasing and positive almost surely (see Lemma \ref{decreasing1}), $\dot \psi$ is bounded above and below by constants on $[a,b]$ and hence for all $\e$ (with sufficiently small absolute
value depending on $[a,b]$, we will need $\e$ to take both positive and negative values in the argument later), we conclude  $g$ is strictly positive. Moreover by \eqref{int20}, 
$\int_0^1 g(x){\rm{d}}x =1$. 
We now verify the area constraint. It is easy to see by Fubini's theorem (see \eqref{fubini} in the proof of Lemma \ref{decreasing1}) that $\int_0^1 (1-x)\dot \psi(x)\ge 1/2+ \alpha$. Thus $$\int_0^1 (1-x)g(x){\rm{d}}x  =\int_0^1 (1-x)\dot \psi(x)+\e \int_0^1 (1-x)h_*(x){\rm{d}}x \ge \frac{1}{2}+\alpha$$
by \eqref{int20} and \eqref{int21}. Now, by Taylor expansion, for all small enough $\e$ (depending on $a$, $b$, $h$ and $\psi$),
\begin{equation}\label{variational20}
\int_0^1 \left[\sqrt{g(x)}-\sqrt{\dot \psi(x)} \right]{\rm{d}}x  = \int_{a}^b \left(\sqrt{\dot \psi(x)}+\e \frac{h_*(x)}{2\sqrt{\dot \psi(x)}}-  \e^2 O\biggl(\frac{h_*(x)^2}{{\dot \psi(x)}^{3/2}}\biggr)\right){\rm{d}}x  , 
\end{equation}
where the constants in the $O$- notation just depend on $[a,b],\psi$ (since $\dot \psi$ is in $[c_3,c_4]$ almost surely for some $c_3<c_4$ on $[a,b]$ by the preceding discussion). Since $J(\psi)= J_{\alpha}$ (and $\e$ can be both positive and negative 
it follows that, 
\begin{equation}\label{ortho}
\int_{a}^b\frac{h_*(x)}{\sqrt{\dot \psi(x)}}{\rm{d}}x =0 .
\end{equation}
Since $\frac{1}{\sqrt{\dot \psi}}$ is bounded above and below on $[a,b]$, one can find a linear function $L(x):=a_1x +b_1$ such that 
\begin{align*}
\int_{a}^{b} \left[\frac{1}{\sqrt{\dot \psi(x)}}-L(x)\right]{\rm{d}}x =0 \,\,\text{and }\,\,  \int_{a}^b \left[\frac{1}{\sqrt{\dot \psi(x)}}-L(x)\right]x{\rm{d}}x =0. 
\end{align*}
Indeed, such $L$ can be found by solving the two linear equations in $a_1$ and $b_1$ given by the above two equations. To see that, for any $a<b$, the equations are non-degenerate, first observe that one can by a change of variable assume that $a=0$ and $b=1$. Then the matrix corresponding to the linear equations is nothing but $\left[\begin{array}{cc}
1/2 &1\\
1/3 &1/2
\end{array}
\right]$  and hence admits a unique solution. Together with \eqref{ortho} this now implies that 
\begin{align*}\int_{a}^b \left[\frac{1}{\sqrt{\dot \psi(x)}}-L(x)\right]P(x){\rm{d}}x =0
\end{align*}
for all polynomials $P(x)$. For this, observe that by Gram-Schmidt orthogonalization any polynomial $P$ can be decomposed as  $P(x)=L'(x)+h_*(x)$ where $L'$ is a linear function and $h_*$ is a polynomial satisfying \eqref{int20} and \eqref{int21}. Since polynomials are dense in $\mathbb{L}_2[a,b]$ by the Stone-Weierstrass theorem, this implies 
$$\frac{1}{\sqrt{\dot \psi(x)}}-L(x)=0~\text{a.e. on}~[a,b].$$

Since  $\frac{1}{\sqrt{\dot \psi}}$ is a linear function on all intervals $[a,b]$, it follows  that it has to be the same linear function on all such intervals and hence $\frac{1}{\sqrt{\dot \psi(x)}}=L(x)$ for some linear function $L(x)$ almost everywhere on the entire interval $[0,1]$.
Thus $\dot \psi(x)=\frac{d}{(1+cx)^2}$ for some constants $c,d$. Integrating and using that $\psi(0)=0$ and $\psi(1)=1$ (see Lemma \ref{l:ac}) we are done. 
\end{proof}

The next proposition completes the identification of $\psi_{\alpha}$ by showing that the constant $c_{\alpha}$ is given by \eqref{implicit1}.

\begin{proposition}
\label{p:calpha}
The constant $c_{\alpha}$ is given implicitly by the following equation:
$$\frac{1+c_{\alpha}}{c_{\alpha}}\left[1-\frac{\log(1+c_{\alpha})}{c_{\alpha}}\right]=\frac{1}{2}+\alpha$$
where $c_{\alpha}$ is the constant in Theorem \ref{solution1} corresponding to $\alpha$.
\end{proposition}


\begin{proof}
Clearly we have
\begin{align*}
\int_{0}^1 \psi(x){\rm{d}}x =\int_{0}^1 \frac{(1+c)x}{1+cx}=\frac{1+c}{c}\left[1-\frac{\log(1+c)}{c}\right]=\frac{1}{2}+\beta
\end{align*}
for some $\beta=\beta(\alpha)\geq \alpha$. We shall show that $\beta=\alpha$ which will establish the proposition. Suppose not. As we have pointed out,  the left-hand side of this equation is increasing in $c$. Let $c_{*}$ be the unique solution to the equation 
\begin{align*}
\int_{0}^1 \psi(x){\rm{d}}x =\int_{0}^1 \frac{(1+c)x}{1+cx}=\frac{1+c}{c}\left[1-\frac{\log(1+c)}{c}\right]=\frac{1}{2}+\alpha.
\end{align*}
It is easy to check that $\beta> \alpha$ implies $c_{*}<c_{\alpha}$. Define $\psi^{*}$ by 
$\psi^*(x)=\frac{(1+c_*)x}{1+c_*x}$.
Clearly $\psi_{*}\in \mathcal{B}_{\alpha}$. A straightforward computation shows that 
\begin{align}
\label{weight}
J(\psi^*)=\int_{0}^1 \sqrt{\dot \psi(x)}{\rm{d}}x = \sqrt{(1+c_*)}\int_{0}^1\frac{1}{1+c_*x} = \sqrt{1+c_*}\frac{\log(1+c_*)}{c_*}.
\end{align}
The right-hand side here is readily seen to be decreasing in $c_*$ 
and hence $J(\psi^*)>J(\psi)$. This contradicts the optimality of $\psi$, and  completes the proof. 
\end{proof}
Observe that, $J_{\alpha}=\sqrt{1+c_{\alpha}}\frac{\log(1+c_{\alpha})}{c_{\alpha}}$.
This expression is the same as the constant $\mathrm{w}_{\alpha}$ in Theorem \ref{t:lln} which is consistent with our heuristic explanation at the beginning of this section. We prove this theorem next. 

\section{Law of large numbers}
\label{s:lln}
With the preparation from the previous section we now turn to the proof of Theorem \ref{t:lln}. Fix $\alpha\in (0,\frac{1}{2})$ as before. The idea of the proof is as follows. We first show that if a path follows approximately the blow up $\psi_{\alpha, n}$ of the deterministic curve $\psi$ (and also satisfies the area constraint) then it is overwhelmingly likely that the path has length at least $2(\mathrm{w}_{\alpha}-\varepsilon)n$ for $\varepsilon$ arbitrarily small. We further show that any path satisfying the area constraint is extremely likely to have length less than $2(\mathrm{w}_{\alpha}+\varepsilon)n$. We begin with the following trivial but useful lemma that is an immediate consequence of continuity and monotonicity of ${\rm{w}}_{\alpha}$.

%

\begin{lemma} Given any $\alpha \in (0,\frac{1}{2}),$ for all small enough $\e>0,$ there exist $\delta_1,\delta_2 > 0$ such that $|{\rm{w}}_{\alpha}-{\rm{w}}_{\alpha_1}| < \e$ if $|\alpha-\alpha_1|\le \delta_1 $ and $|{\rm{w}}_{\alpha}-{\rm{w}}_{\alpha_1}| > \e$ if $|\alpha-\alpha_1| \ge \delta_2$.
\end{lemma}
Next we need a couple of preparatory lemmas about appropriate discretisations. Let $I_{\delta}=\{0,\delta,2\delta,\ldots,1\}$ be the discretisation of the unit interval; (we will take $\delta$ to be the reciprocal of a positive integer to avoid rounding errors). Consider any non-decreasing function $L: [0,1] \to [0,1]$ that corresponds to an absolutely continuous measure on $[0,1]$: by identifying $L$ with its graph, we shall interpret $L$ as an increasing path on the unit square directed from $(0,0)$ to $(1,1)$. 
Also let $L_{\delta}$ be  the piecewise affine function that agrees on $I_\delta$ with $L$.

\begin{lemma} 
\label{l:approx3}
We have $\left | \int_{0}^1[L(x)-L_{\delta}(x)]{\rm{d}}x \right | \le 2\delta$.
\end{lemma}

\begin{proof}First notice that by integration by parts or by Fubini's theorem, for any absolutely continuous function $f$ on $[0,1],$  $\int_0^1 f(x){\rm{d}}x =\int_0^1 \dot f(x) (1-x){\rm{d}}x $. The lemma now follows from the observation that $\left |\int_{0}^1 x \dot L(x){\rm{d}}x-\int_{0}^1 x \dot L_{\delta}(x){\rm{d}}x\right|\le   2\delta$.
{Let $x_{i}$ be the midpoint of the interval $I_\delta^{(i)}=[i\delta,(i+1)\delta]$. 
Notice that by definition $\dot L_{\delta}$ is constant on $I_\delta^{(i)}$ and is equal to $\frac{1}{\delta}\int_{I_\delta^{(i)}}\dot L(x){\rm{d}}x $. 
Thus 
$$
\sum_{i=0}^{\frac{1}{\delta}-1}\int_{I_{\delta}^{(i)}} x_i \dot L(x) {\rm{d}}x=\sum_{i=0}^{\frac{1}{\delta}-1}x_i \int_{I_{\delta}^{(i)}}  \dot L(x){\rm{d}}x=\sum_{i=0}^{\frac{1}{\delta}-1}x_i \int_{I_{\delta}^{(i)}}  \dot L^{\delta}(x){\rm{d}}x 
=\sum_{i=0}^{\frac{1}{\delta}-1}\int_{I_{\delta}^{(i)}} x \dot L^{\delta}(x)=\int_0^1 x \dot L^{\delta}(x)
$$
The proof now follows by observing $|\int_{0}^1 x \dot L(x){\rm{d}}x -\int_{0}^1 x \dot L^{\delta}(x){\rm{d}}x | \le \sum_{i=0}^{\frac{1}{\delta}-1}\int_{I_{\delta}^{(i)}} |x - x_i| \dot L(x){\rm{d}}x\le \delta$.
}
\end{proof}
We now discretise the vertical direction as well. 
Let us choose $\eta\ll \delta$ (to be specified exactly later),  and  let $B_{\delta,\eta}= I_{\delta} \times I_{\eta}$. Given $L$ as before let us now set, 
$L_{\delta,\eta}$ to be the piecewise linear curve determined by the points  $(i\delta, \eta \lfloor \frac{L(i\delta)}{\eta} \rfloor)$.
Clearly for all $x,$ $|\dot L_{\delta}-\dot L_{\delta,\eta}| \le \frac{2\eta}{\delta},$ and hence for $\eta\le \delta^2,$
\begin{equation}\label{area14}
\bigg\vert \,  \int_{0}^1[L(x)-L_{\delta,\eta}(x)]{\rm{d}}x \, \bigg\vert \le 4\delta. \end{equation}


Before proceeding further we make a few comments about notation. Recall that till now,  our underlying noise space has been a point process $\Pi$ on $\mathbb{R}^2$ and we have been concerned about geodesics in $[0,n]^2$. However in the last section while solving variational problems we switched to a normalized picture where every path lives in $[0,1]^2$. We will continue with this convention throughout this section. Equivalently the noise space can be thought of as a Poisson point process of intensity $n^{2}$ in $[0,1]^2$. In an abuse of notation,  we will define $L(u,v)$ for two points $u,v \in [0,1]^2$ to be the length of the geodesic $\gamma(u,v)$ between $u$ and $v$ in the same noise space. In light of  properties of Poisson process under scaling, 
{Moreover, all the estimates regarding $\gamma(u,v)$ stated in the previous sections continue to hold after the appropriate variable change. We omit further elaboration.}  
Next, we state a uniform version of  Theorem \ref{t:moddev} in the large deviation regime. 
 
 
 
\begin{lemma}
\label{uni100} Fix $\delta,\eta$ as above.
There exists $c=c(\delta,\eta)$ such that, simultaneously for all $x\in I_{\delta}$ and $y_2> y_1 \in I_{\eta}$ for all large $n$:
\begin{enumerate}
\item With probability at least $1-e^{-cn},$ $$L((x,y_1),((x+\delta),y_2))- 2 n \sqrt{\delta(y_2-y_1)}\le \eta \delta n,$$ 
\item With probability at least $1-e^{-cn^2},$ $$L((x,y_1),((x+\delta),y_2))- 2 n \sqrt{\delta(y_2-y_1)}\ge - \eta \delta n.$$  
\end{enumerate}
\end{lemma}

\begin{proof} The proof follows from the large deviation probabilities for the  length of unconstrained geodesics in \cite{sepLDP}. 
In particular the theorem states that for any $x,y_1,y_2$ as in the statement of the theorem,  
\begin{align*}
\P(L((x,y_1),((x+\delta),y_2))- 2 n \sqrt{\delta(y_2-y_1)}\ge \eta \delta n)&\le e^{-cn},\\
\P(L((x,y_1),((x+\delta),y_2))- 2 n \sqrt{\delta(y_2-y_1)} -\eta \delta n)&\le e^{-cn^2},
\end{align*}
for some constant $c=c(\delta,\eta)$.
A union bound over points $x,y_1,y_2$, the total number of which is $\frac{1}{\delta \eta^2}$, completes the proof.
\end{proof}

We are now ready to prove Theorem \ref{t:lln}. We start by showing the upper bound. Recall that $L_{\alpha}(n)$ denotes the length of the constrained geodesic.

\begin{proposition}
\label{p:llnub}
Fix $\varepsilon>0$. There exists $c=c(\alpha,\varepsilon)>0$ such that with probability at least $1-e^{-cn}$ we have 
$$L_{\alpha}(n) \leq 2(\mathrm{w}_{\alpha}+\varepsilon)n.$$ 
\end{proposition}

\begin{proof}
Recall that for any increasing path $\gamma,$ we denote its length or the number of points of the Poisson process it passes through by $|\gamma|$.
Let us now notice that $$|\gamma|\le \sum_{i\in I_{\delta}}L\biggl(\bigl(i\delta, \gamma_{\delta,\eta}(i\delta)\bigr), \bigl((i+1)\delta, \gamma_{\delta,\eta}((i+1)\delta)+\eta\bigr)\biggr).$$ 
Thus by the previous lemma, with probability at least  $1-e^{-cn},$ we have for all increasing paths, $\gamma$ (from $(0,0)$ to $(1,1)$ in the rescaled space)
\begin{align*}
|\gamma|&\le \sum_{i\in I_{\delta}} 2\delta \sqrt{\frac{\gamma_{\delta,\eta}((i+1)\delta)+\eta)-\gamma_{\delta,\eta}(i\delta)}{\delta}}n +2\eta n, \\
&= \sum_{i\in I_{\delta}} 2n \delta \sqrt{\dot \gamma_{\delta,\eta}+\frac{\eta}{\delta}} +2\eta n.
\end{align*}
An argument involving truncating at $\dot \gamma_{\delta,\eta}< \sqrt{\frac{\eta}{\delta}}$ shows that, with probability at least  $1-e^{-cn}$, for all increasing paths $\gamma,$
$$\sum_{i\in I_{\delta}} n\delta \sqrt{\dot \gamma_{\delta,\eta}+\frac{\eta}{\delta}}\le n\sum_{i} \delta \sqrt{\dot \gamma_{\delta,\eta}}  +{\left(\frac{\eta}{\delta}\right)}^{1/4}O(n).
$$

Thus, for $\eta < \delta^5$, with probability at least $1-e^{-cn},$ for all $\gamma,$ 
\begin{equation}\label{approx34}
|\gamma|\le n \int_{0}^1 2\sqrt{\dot \gamma_{\delta,\eta}}(x){\rm{d}}x +O(\delta n).
\end{equation} 

Now let $\gamma$ be an 
increasing path that traps area at least $(1/2+\alpha)$ (which corresponds to area $(1/2+\alpha)n^2$ in the unscaled model).  By \eqref{area14},  it follows that $\int_0^1 \gamma_{\delta,\eta}(x){\rm{d}}x \ge \frac{1}{2}+\alpha- O(\delta)$ and hence, by continuity of $\mathrm{w}_{\alpha}$, we have $\int_{0}^1 \sqrt{\dot \gamma_{\delta,\eta}}(x){\rm{d}}x\le ({\rm w}_{\alpha}+\e/2)$ by choosing $\delta$ sufficiently small. By choosing $\delta$ suitably small, this implies that, with probability at least $1-e^{-cn}$, the bound $|\gamma|\le 2( w_{\alpha}+\e)n$ for such increasing paths $\gamma$. This completes the proof of the proposition.
\end{proof}

We now show the corresponding lower bound. 

\begin{proposition}
\label{p:llnlb}
Fix $\varepsilon>0$. There exists $c=c(\alpha,\varepsilon)>0$ such that, with probability at least $1-e^{-cn^2}$, we have 
$$L_{\alpha}(n) \geq 2(\mathrm{w}_{\alpha}-\varepsilon)n.$$ 
\end{proposition}

 \begin{proof}
The derivation of this lower bound is straightforward:  we will take a discretisation of $\psi_{\alpha}$ and then use the continuity of $\dot \psi$.
Fix $\e>0$. Choose $\theta>0$ and consider $\psi_*=\psi_{\alpha+\theta}$. Then, by definition, 
$
\int_{0}^1 \psi_*(x){\rm{d}}x\ge \frac{1}{2}+\alpha+\theta .
$
Choose $\theta$ small enough that 
$
{\rm w}_{\alpha+\theta}=\int_{0}^1 \sqrt{\dot \psi_*(x)}{\rm{d}}x\ge {\rm w}_{\alpha}-\frac{\e}{2}.
$

By using the continuity of $\psi$ and $\dot \psi$, we see that, for given $\e$, there exist sufficiently small choices of  $\delta,\eta,$ such that the piecewise affine function $\psi_{*,\delta, \eta}$ interpolating between the points $u_i:=(i\delta,\eta\lfloor\frac{\psi_*(i\delta)}{\eta}\rfloor)$
satisfies 
 \begin{align*}
 \int_{0}^1 \psi_{*,\delta,\eta}(x){\rm{d}}x\ge \frac{1}{2}+\alpha+\frac{\theta}{2},&
 \int_{0}^1 \sqrt{\dot \psi_{*,\delta,\eta}}(x){\rm{d}}x\ge {\rm w}_\alpha-\frac{3\e}{4}.
 \end{align*}
Now let $\gamma_{(i,i+1)}$ 
be the geodesic between the points $u_i$ and $u_{i+1}$.  Let $\gamma$ be the path obtained by concatenating the paths $\gamma_{(i,i+1)}$ for $i=0,\ldots,\frac{1}{\delta}-1$ 
Using Lemma \ref{uni100}, it follows that, with probability at least $1-e^{-cn^2}$, 
$$
|\gamma|\ge \sum_{i\in I_{\delta}} 2n\delta \sqrt{\dot \psi_{*,\delta,\eta}} -2n\eta \geq 2(\mathrm{w}_{\alpha}-\e)n ,
$$
where we take $\eta$ and $\delta$ sufficiently small. Now  by
Lemma \ref{l:approx3}, we find that 
$$\left |\int_{0}^{1} \gamma(x){\rm{d}}x - \int_{0}^1 \psi_{*,\delta,\eta}(x){\rm{d}}x \right|\le 2\delta,$$ since $\gamma$ and $\psi_{*,\delta,\eta}$ agree  on $I_{\delta}$.
Hence, for $\delta$ small enough,
$
\int_{0}^1 \gamma(x){\rm{d}}x \ge \frac{1}{2}+\alpha
$.
Thus, $L_{\alpha,n}$ is at least as large as $|\gamma|$. Hence, we are done.
\end{proof}

\begin{proof}[Proof of Theorem \ref{t:lln}]
Combining Proposition \ref{p:llnub} and Proposition \ref{p:llnlb}, we complete the proof of Theorem \ref{t:lln} by noting that an upper tail bound of the form $\P(L_{\alpha}(n)\geq k)\leq e^{-k}$ for all $k\geq 5n^2$ is easily obtained by bounding the upper tail of the Poisson-distributed  total number of points. 
\end{proof}

\subsection{Law of large numbers for the geodesic}

We now prove  Theorem \ref{lln2}. The proof is by contradiction:  if ${\rm dist}(\Gamma_{\alpha},\psi_{\alpha})\ge \Delta$, and the noise space is such that the events listed in Lemma \ref{uni100} 
occur, then we would be able to construct an absolutely continuous function $h$ that traps area at least  $1/2+\alpha$ and for which   $\int_0^1 \sqrt{\dot h}(x){\rm{d}}x>\int_0^1\sqrt{\dot \psi_{\alpha}}(x){\rm{d}}x $. This would contradict extremality  of $\psi_{\alpha}$.
To show the above inequality, we will employ a concavity argument. 
For notational brevity, let $\gamma=\Gamma_{\alpha}$ and $\psi=\psi_{\alpha}$.  For $\delta,\eta>0$, recall the definition of $\gamma_{\delta,\eta}$ from the previous section. 
We will consider the function $h=\frac{\gamma_{\delta,\eta}+\psi}{2}$.
Fixing parameters $\e,\delta$ and $\eta\le \delta^5,$
we will restrict attention to the event $\mathcal{A}$,
on which:
\begin{enumerate}
\item $|\gamma|\le 2n \int_{0}^1 \sqrt{\dot \gamma_{\delta,\eta}}(x){\rm{d}}x +O(\delta n),$ \\
\item $2({\rm w}_\alpha -\e)n \le |\gamma|\le 2({\rm w}_\alpha +\e)n$.
\end{enumerate}
By Propositions \ref{p:llnub} and \ref{p:llnlb}, and \eqref{approx34}, $\mathcal{A}$ occurs with probability at least $1-e^{-cn}$. 
Observe that Proposition \ref{solution1} implies that $\dot \psi$ is bounded away from zero and infinity on $[0,1]$.  Thus, $\dot \psi$ lies in  $[c_1,C_2]$ where $0<c_1<C_2 < \infty$.
Suppose that $\displaystyle{\sup_x|\gamma(x)-\psi(x)|\ge \Delta}$.  We now claim that  if $\delta,\eta$ are small enough compared to $\Delta$,  then  there exists $y\in[0,1]$ such that    
\begin{equation}\label{conseq203}
|\gamma_{\delta,\eta}(y)-\psi(y)|\ge \frac{\Delta}{2}.
\end{equation}
To see this, we choose $y\in[0,1]$ such that $|\gamma(y)-\psi(y)|\ge \Delta$. Let $y_1\in I_{\delta}$  be such that $y_1\le y< y_1+\delta$.
Now we consider two cases: 
\begin{enumerate}
\item 
If $\gamma(y)\ge \psi(y)+\Delta$, then
 \begin{align*}
\gamma_{\delta,\eta}(y_1+\delta) \ge \gamma(y)-\eta\ge \psi(y)+\Delta-\eta &\ge \psi(y_1+\delta)+\Delta-O(\delta), \text{ and hence,}\\
\gamma_{\delta,\eta}(y_1+\delta)& \ge  \psi(y_1+\delta)+\Delta-O(\delta).
\end{align*}
 These implications, excepting the last, follow by definition. The last implication is due to  $\psi$ having Lipschitz constant at most $C_2$.
\item If $\psi(y)\ge \gamma(y)+\Delta$, then, similarly,
 \begin{align*}
 \psi(y_1)\ge \psi(y)-O(\delta)&\ge \gamma(y)+\Delta-O(\delta)\ge \gamma_{\delta,\eta}(y_1)+\Delta-O(\delta) \text{ and hence,}\\
  \psi(y_1)&\ge \gamma_{\delta,\eta}(y_1)+\Delta-O(\delta).
\end{align*}
 \end{enumerate}
Thus, by \eqref{conseq203}, 
$$\int_0^1|\dot \gamma_{\delta,\eta}(x)-\dot \psi(x)|{\rm{d}}x\ge  \left|\int_{0}^y[\dot \gamma_{\delta,\eta}(x)-\dot \psi(x)]{\rm{d}}x \right| \ge \frac{\Delta}{2}.$$
Also, note that, since both $\psi$ and $\gamma_{\delta,\eta}$ start at $(0,0)$ and end at $(1,1)$,  
$\int_0^1\dot \gamma_{\delta,\eta}(x){\rm{d}}x =\int_0^1\dot \psi(x){\rm{d}}x$.
Hence, 
\begin{equation}\label{algebra102}
\int_0^1 [\dot \psi(x)- \dot \gamma_{\delta,\eta}(x)]\mathbf{1}(S){\rm{d}}x \ge \Delta/4,
\end{equation}
where $S$ denotes $\{x\in [0,1]: \dot \psi(x)\ge \dot \gamma_{\delta,\eta}\}$.   Moreover, on  $S$, 
\begin{equation}\label{bound303}
\dot \gamma_{\delta,\eta}\le \dot \psi \le C_2
\end{equation} 
for some $C_2 >1$.
Observing that $\dot h=\frac{\dot \gamma_{\delta,\eta}+{\dot \psi}}{2}$, 
 we see that the  concavity of  $x\to \sqrt x$ implies that $\sqrt{\dot h}-\frac{\sqrt{\dot \gamma_{\delta,\eta}}+\sqrt{\dot \psi}}{2},$ is  non-negative.  
Furthermore, simple algebra shows that, for all $x\in S$,
\begin{align}\label{algebra101}
\sqrt{\dot h}(x)-\frac{[\sqrt{\dot \psi}(x)+\sqrt{\dot \gamma_{\delta,\eta}}(x)]}{2} &= \frac{(\sqrt{\dot \psi}(x)-\sqrt{\dot \gamma_{\delta,\eta}}(x))^2}{4[\sqrt{\dot h}(x)+\frac{[\sqrt{\dot \psi}(x)+\sqrt{\dot \gamma_{\delta,\eta}}(x)]}{2}]}\stackrel{\eqref{bound303}}{\ge}  \frac{(\sqrt{\dot \psi}(x)-\sqrt{\dot \gamma_{\delta,\eta}}(x))^2}{8 C_2},\\
\nonumber
&\ge \frac{({\dot \psi}(x)-{\dot \gamma_{\delta,\eta}}(x))^2}{8C_2[\sqrt {\dot \psi}(x)+\sqrt{\dot \gamma_{\delta,\eta}}(x)]^2}\stackrel{\eqref{bound303}}{\ge} \frac{({\dot \psi}(x)-{\dot \gamma_{\delta,\eta}}(x))^2}{100 C_2^2}.
\end{align} 

The proof is now completed by observing that a length gain has been realized:\begin{align*}
\int_0^1 \left[\sqrt{\dot h}(x)-\frac{[\sqrt{\dot \psi}(x)+\sqrt{\dot \gamma_{\delta,\eta}}(x)]}{2}\right]{\rm{d}}x& \ge \int_0^1 \left[\sqrt{\dot h}(x)-\frac{[\sqrt{\dot \psi}(x)+\sqrt{\dot \gamma_{\delta,\eta}}(x)]}{2}\right] \mathbf{1}(S){\rm{d}}x \\
&\stackrel{\eqref{algebra101}}{\ge} \int_0^1 \frac{({\dot \psi}(x)-{\dot \gamma_{\delta,\eta}}(x))^2}{100 C_2^2}\mathbf{1}(S){\rm{d}}x \\
&\ge  \frac{1}{100C_2^2} \left(\int |({\dot \psi}(x)-{\dot \gamma_{\delta,\eta}}(x)|\mathbf{1}(S){\rm{d}}x \right)^2\\
 &\stackrel{\eqref{algebra102}}{\ge} \frac{1}{2000C_2^2} \Delta^2. 
\end{align*}
On the event $\mathcal{A},$ we see, by means of this inequality and the definition of  $\mathcal{A}$, that if  $\displaystyle{\sup_x|\gamma(x)-\psi(x)|\ge \Delta} $ holds, 
then $$\int_0^1\sqrt{\dot h}(x){\rm{d}}x \ge [({\rm w}_\alpha-\e)+\frac{1}{C^2} \Delta^2]$$ for some constant $C$ that  depends only on $\alpha$.
Also, both  $\int_{0}^1\psi(x){\rm{d}}x $ and $\int_{0}^1\gamma(x){\rm{d}}x $ are at least $\frac{1}{2}+\alpha$ by definition,  and hence  by \eqref{area14}, we have 
$$
\int_0^1 h(x){\rm{d}}x =\frac{\int_0^1 \psi(x){\rm{d}}x +\int_0^1 \gamma_{\delta,\eta}(x){\rm{d}}x }{2}\ge \frac{1}{2}+\alpha-\e/2,
$$ where $\e$ can be made arbitrarily small by choosing $\delta$ and hence $\eta$ small enough. This inference contradicts the continuity of ${\rm w}_{\alpha}$ in $\alpha$.
Hence we are done.
\qed


We now use similar arguments as those employed to prove the law of large numbers in order to establish a variant of Theorem \ref{flat0} that is also uniform in $\alpha$. This particular variant will be crucial in the proof of Theorem \ref{t:mflub}.

\begin{theorem}
\label{flat1} 
Fix $0< \alpha_1<\alpha_2 <\frac{1}{2}$. For any small enough $\delta>0$, there exist $\gamma>0$ and $c>0$ such that, with probability at least $1-e^{-cn}$, all $\delta$-interior facets of $\Gamma_{\alpha,n}$ for which $\alpha\in[\alpha_1,\alpha_2]$ make an angle with the $x$-axis that lies $(\omega,\pi/2-\omega)$.
\end{theorem}

\begin{proof} 
Recall that the proof of  Theorem \ref{flat0}  
used Theorem \ref{lln2} and the strict convexity of the function $\psi_\alpha$. Since $\psi_{\alpha}$ is uniformly convex for all $\alpha\in [\alpha_1,\alpha_2]$, this proof will  be complete, using the same arguments as in the proof of Theorem \ref{flat0}, once we prove the following uniform version of Theorem \ref{lln2}: for any $\Delta>0$, with probability at least $1-e^{-cn}$,  $$\sup_{\alpha\in[\alpha_1,\alpha_2]}\sup_x|\Gamma_{\alpha}(x)-\psi_{\alpha}(x)|\le \Delta,$$ where $c$ depends on $\Delta$ and the interval $[\alpha_1,\alpha_2]$.
We fix an $\e$ to be specified later and discretise the interval $[\alpha_1,\alpha_2]$ to obtain the set $\mathcal{B}=\{\alpha_1, \alpha_1+\e, \alpha_1+2\e,\ldots,\alpha_2\}$.
Fixing parameters $\e,\delta$ and $\eta\le \delta^5$,
again as before we restrict attention to the event $\mathcal{A}$ on which, for all $\alpha \in \mathcal{B}$,
\begin{enumerate}
\item $|\Gamma_{\alpha}|\le 2 n \int_{0}^1 \sqrt{\dot{ \Gamma}_{\alpha,\delta,\eta}}(x){\rm{d}}x +O(\delta n)$; 
\item and $2({\rm w}_\alpha -\e)n \le |\Gamma_{\alpha}|\le 2({\rm w}_\alpha +\e)n$.
\end{enumerate}
Similarly to a previous argument, by Propositions \ref{p:llnub} and \ref{p:llnlb} and \eqref{approx34}, followed by a union bound, $P(\mathcal{A})\ge 1-e^{-cn}$. 
Now if $\delta$ and $\eta$ are chosen to be sufficiently small depending on $\Delta,$  by the previous result, Theorem \ref{lln2}, and  a simple union bound,  we obtain $$\sup_{\alpha\in \mathcal{B}}\sup_x|\Gamma_{\alpha}(x)-\psi_{\alpha}(x)|\le \Delta.$$
The proof will proceed along the same lines as the proof of Theorem \ref{lln2} did.  With the aim of  arriving at a contradiction, let $\alpha\in [\alpha_1+i\e,\alpha_1+(i+1)\e]$ be such that 
\begin{equation}\label{violate}
\sup_x|\Gamma_{\alpha}(x)-\psi_{\alpha}(x)|> \Delta.
\end{equation}
Clearly, $|\Gamma_{\alpha_1+(i+1)\e}|\le |\Gamma_{\alpha}|\le |\Gamma_{\alpha_1+i\e}|$ holds by definition.
Since  $\psi_{\beta}$ is a continuous function of  $\beta$ in the supremum norm, we see that, for small enough $\e$, $$\displaystyle{\sup_x|\Gamma_{\alpha}(x)-\psi_{\alpha_1+i\e}(x)|\ge \Delta/2}.$$
Now,  as we argued in the proof of  Theorem \ref{lln2}, this implies that 
\begin{align}\label{dev100}
\sup_x|\Gamma_{\alpha,\delta,\eta}(x)-\psi_{\alpha_1+i\e}(x)|\ge \Delta/2.
\end{align}
Note then that
$$
|\Gamma_{\alpha}|\ge |\Gamma_{\alpha_1+(i+1)\e}|\ge  2n({\rm{w}}_{\alpha_1+(i+1)\e}-\e) \ge 2n({\rm{w}}_{\alpha}-\e_1),
$$ 
where where $\e_1$ can be made small enough by choosing $\e$ small enough. The first inequality follows by definition, and the second by the occurrence of $\mathcal{A}$.

Thus, using \eqref{approx34}, we find that  
$$\int_{0}^1 \sqrt{\dot \Gamma_{\alpha,\delta,\eta}}(x){\rm{d}}x\ge {\rm{w}}_{\alpha}-\e_1.$$ 
This along with \eqref{dev100} allows us to  apply the concavity argument that appears in the proof of Theorem \ref{lln2}:  by choosing $\e,\delta,\eta$ much smaller than $\Delta$, we thus contradict the continuity of ${\rm w }_{\beta}$ at $\beta=\alpha$.
\end{proof}

\section{Lower Bound for Scaling Exponents}
\label{s:lb}
In this section we provide proofs of the lower bounds on ${\rm MFL}(\Gamma_n)$ and ${\rm MLR}(\Gamma_n)$, i.e., we prove Theorem \ref{t:mfllb} and Theorem \ref{t:mlrlb}. Most of the work goes into proving the ${\rm MFL}$ lower bound Theorem~\ref{t:mfllb}, since the lower bound for local roughness is a reasonably easy corollary of Theorem \ref{t:tftail}. Let $\alpha\in (0,\frac{1}{2})$ and $\varepsilon>0$ be fixed for the rest of the section. Let $\mathcal{A}_{\varepsilon}$ denote the event that ${\rm MFL}(\Gamma_n)\leq n^{3/4-\varepsilon}$. We shall show that the event $\mathcal{A}_{\varepsilon}$ is extremely unlikely. We start with an overview of the proof. We shall need a geometric definition.

\begin{definition}
\label{d:reg}
Let $C>1,\kappa>1$ be given constants. A sequence of points $u_0\preceq u_1 \preceq \cdots \preceq u_k$ is called a $(C,\kappa)$-\textbf{regular} sequence if the following conditions hold.
\begin{enumerate}
\item[(i)] The union of line segments joining $u_{i}$ to $u_{i+1}$ for $i=0,1,\ldots,k-1,$ is convex.
\item[(ii)] The gradient of the all the line segments joining $u_i$ to $u_{i+1}$ is $\in (\frac{1}{\kappa}, \kappa)$.
\item[(iii)]The distance between the first and last point in the sequence i.e., $|u_k-u_0|\in (\frac{1}{C}n^{3/4-\varepsilon/2}, Cn^{3/4-\varepsilon/2})$.
\item[(iv)] The distance between the consecutive points in the sequence is small: $|u_i-u_{i+1}|\leq n^{3/4-\e}$.
\item[(v)] Let $\theta_1$ and $\theta_2$ be the angles that the line segments $(u_0,u_1)$ and $(u_{k-1},u_{k})$ make with the positive $x$-axis (clearly $\theta_1\geq \theta_2$ by hypothesis). Then $\theta_1-\theta_2 \leq 100C^2 n^{-1/4-\e/2}$.
%
\end{enumerate}
\end{definition}

See Figure \ref{f:corner} for an illustration of this definition. 

\begin{figure}[h] 
\centering
\begin{tabular}{cc}
\includegraphics[width=0.4\textwidth]{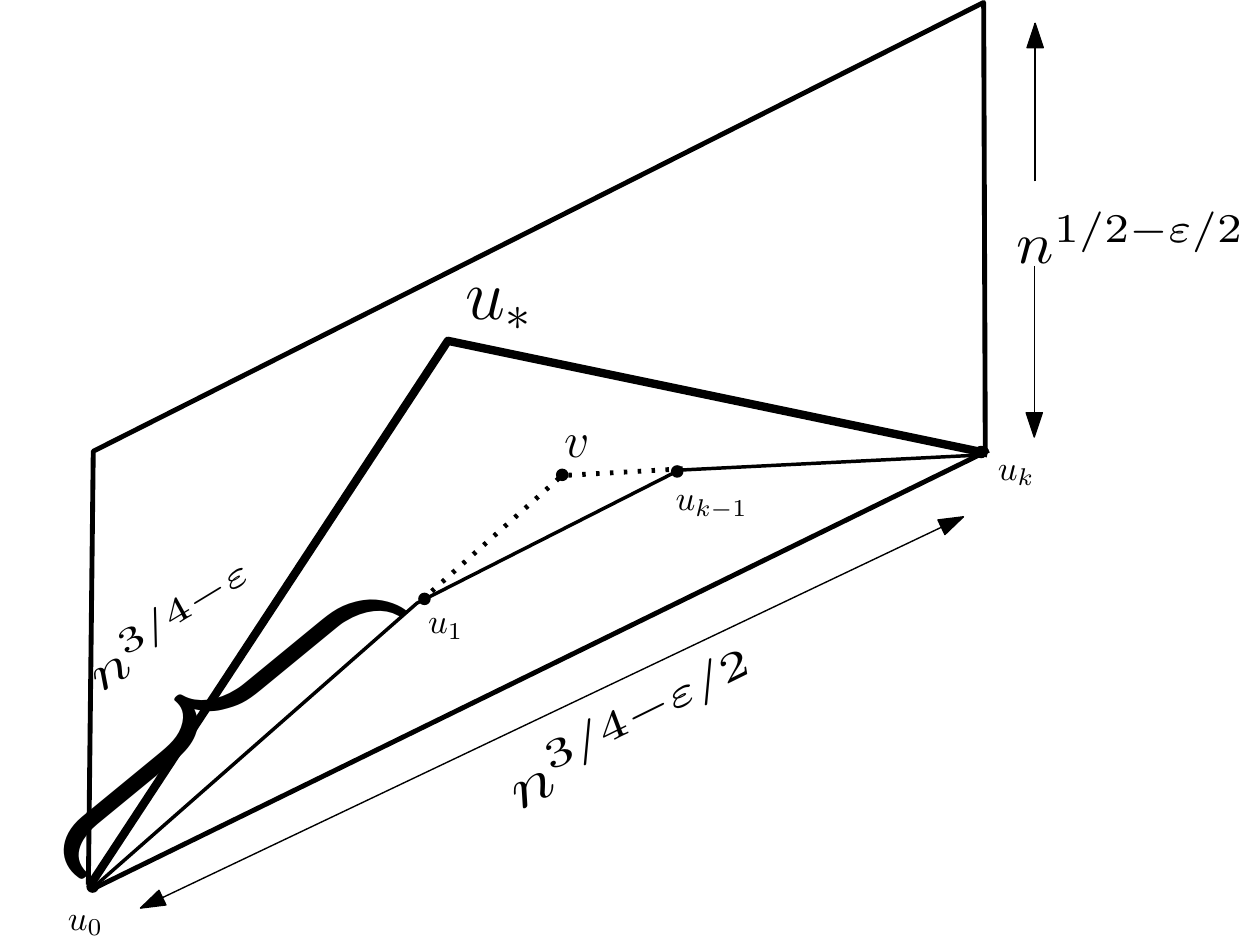} &\includegraphics[width=0.3\textwidth]{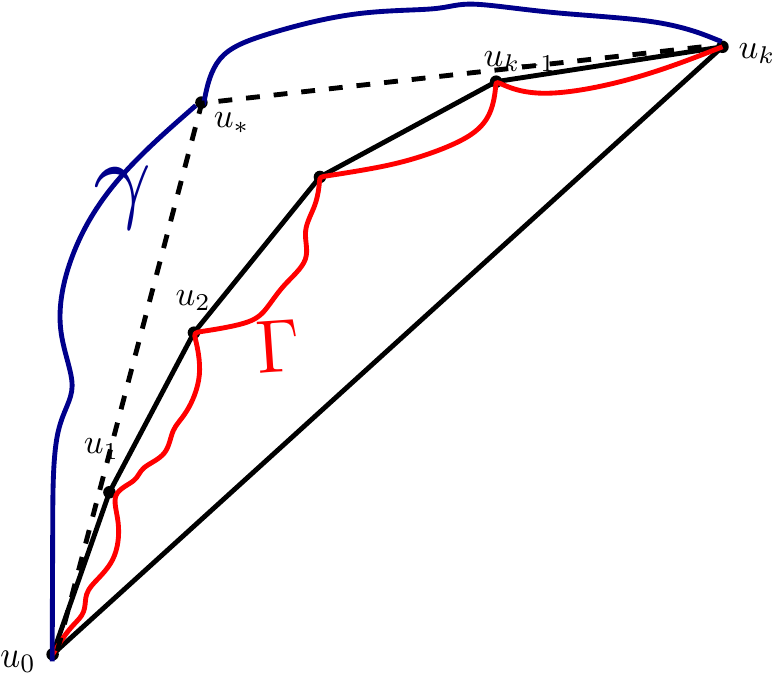} \\
(a) & (b)
\end{tabular}
\caption{ (a) A sequence of regular corners $u_0,u_1,\ldots, u_k$. The triangle $T$ is an isosceles triangle formed by vertices $u_0, u_k$ and $u_*$ such that the adjacent sides lie above the piecewise linear path passing through the points $u_0,u_1,\ldots, u_k$.  By definition of regularity, the triangle $T$ is contained in the parallelogram $R$ formed by the vertices $u_0, v_0,v_k$ and $u_k$ which has vertical height $n^{1/2-\varepsilon/2}$. Observe that the height of this parallelogram is much smaller than the transversal fluctuation of paths between $u_0$ and $u_k$. 
(b) On $\mathcal{A}_{\varepsilon}$, there is a regular sequence of corners $u_0, u_1,\ldots, u_{k}$ and $\Gamma$ is the best constrained path. We show that there is an alternative path $\gamma$ between $u_0$ and $u_1$ above the triangle $T$ that is extremely likely to be strictly larger in length than the restriction of $\Gamma$ between $u_0$ and $u_{k}$. Then the path obtained by replacing $\Gamma$ with $\gamma$ between $u_0$ and $u_k$ has a larger length and traps a larger area, thus leading to a contradiction.
}
\label{f:corner}
\end{figure}


Before proceeding let us try to motivate the item (v) of the above definition. Consider the least concave majorant of an increasing path from $(0,0)$ to $(n,n)$ as described in Definition \ref{d:mlrmfl}  
The total change of angle made by the  line segments constituting the concave majorant with the $x$-axis is roughly around $\pi/2$ over a length of around $n$. So over a distance of around $n^{3/4-\e/2}$, the change of angle should be around $n^{1/4-\e/2}$ on average. The item (v) of the above definition asserts that on a regular sequence the change of angle is not much more than this average. 

The following basic geometric consequence will be useful.

\begin{lemma}
\label{l:reggeo}
Let $S=\{u_0,u_1,\ldots, u_k\}$ be a $(C,\kappa)$-regular sequence for some $C,\kappa>1$. Let $L_{S}$ denote the union of line segments joining the consecutive points of $S$. Let $R$ denote the parallelogram with two vertical sides of length 
$n^{1/2-\varepsilon/2}$ whose bottom side is the line segment joining $u_0$ and $u_k$; see Figure \ref{f:corner}(a). Then there exists a point $u_*$ in $R$ such that the triangle $T$ formed by the points $u_0,u_k, u_*$ is an isosceles triangle (with the sides adjacent to $u_*$ being of equal length) contained in $R$ such that the two sides of $T$ adjacent to $u_*$ lie above the piecewise affine curve obtained by joining the consecutive points of $S$. 
\end{lemma} 
\begin{proof}
Consider the angles $\omega_1$ and $\omega_2$ made by the line segments $(u_0,u_1)$ and $(u_{k-1},u_k)$ respectively with the line segment $(u_0,u_k)$.  Also, let $v$ be the point of intersection obtained by extending the line segments  $(u_0,u_1)$ and $(u_{k-1},u_k)$. By convexity, the piecewise affine curve obtained by joining the points of $S$ lies inside the triangle $(u_0,u_k,v)$. Moreover, it follows from elementary geometric arguments that $\theta_1-\theta_2=\omega_1+\omega_2$. Now let us consider the isosceles triangle $(u_0,u_k,u_*)$ where $(u_0,u_k)$ forms the base and $\theta_1-\theta_2$ is the value of the two equal angles. Since $\theta_1-\theta_2$ is at least as big as $\omega_1$ and $\omega_2$, it follows that the triangle $(u_0,u_k,v)$ is contained inside the triangle $(u_0,u_k,u_*)$.  Moreover, the latter is clearly contained in a parallelogram of height $O(n^{3/4-\e/2}n^{-1/4-\e/2})$ where the constant in the $O(\cdot)$ notation depends on $\kappa$. Since for all large enough $n$,  $n^{3/4-\e/2}n^{-1/4-\e/2}\ll n^{1/2-\e/2}$, we are done. 
\end{proof}

Let us now record the main steps of the proof of Theorem \ref{t:mfllb}. Let $\Gamma$ denote the constrained geodesic. Recall that $\Gamma^{*}_{\alpha,n}$ denotes least concave majorant of $\Gamma$ and is an union of the facets as described in Definition \ref{d:mlrmfl}.  

\textbf{Step 1:} We shall fix a $(C,\kappa)$-regular sequence $S=\{u_0,u_1,\ldots, u_k\}$, and consider the best path~$\gamma_{S}$ from $u_0$ to $u_{k}$ that passes through all the points in the sequence. We shall show that with very high probability the length of this path is much smaller than the length of the best path between $u_0$ and $u_{k}$. The reason that this event is likely is the following: because the lengths of the segments $\{u_{i}, u_{i+1}\}$ are much smaller than that of the segment $\{u_{0}, u_{k}\}$, it will turn out that the path $\gamma$ will have a much smaller transversal fluctuation compared to the typical geodesic between $u_0$ and $u_{k}$. Thus, by Theorem~\ref{t:constrained},  it is extremely likely that this path has a much smaller length. 

\textbf{Step 2:} Next we will show that there exists a path between $u_0$ and $u_{k}$ that lies above the line segments joining the consecutive points of $S$ and whose length is comparable to that of the geodesic between $u_0$ and $u_{k}$. This will follow from the definition of regular sequence together with an application of Lemma \ref{l:reggeo} and Theorem \ref{baikrains}.  

\textbf{Step 3:} Finally, we will show that, on the event $\mathcal{A}_{\e}$, it is overwhelmingly likely that there exists a $(C,\kappa)$-regular sequence made out of consecutive corners of $\Gamma^{*}_{\alpha,n}$. Once we establish this, we will arrive at a contradiction: using the first two steps, we will construct a path that traps more area than does $\Gamma$ and that also has a greater length.

The next proposition treats the first step.

\begin{proposition}
\label{p:step1}
Let $S=\{u_0,u_1,\ldots, u_k\}$ be a $(C,\kappa)$-regular sequence. For $i=0,1,\ldots,k-1,$ let $\gamma_{i}$ be the geodesic 
between $u_{i}$ and $u_{i+1}$. Let $\gamma_{S}$ denote the concatenation of the paths $\gamma_{i}$. Then there exists $c>0$ such that, with probability at least $1-e^{-n^{c}}$, we have 
$$|\gamma_{S}|\leq \E L(u_0,u_k)-n^{1/4-\e/12}.$$
\end{proposition}

\begin{proof}
Let $R'$ be the parallelogram with sides parallel to the sides of $R$ such that $u_0$ and $u_{k}$ are the midpoints of the vertical sides of $R'$; and the height of the vertical sides is $4n^{1/2-\e/2}$.  For each $i$, let $R_i$ denote the parallelogram with the following properties:
\begin{enumerate}
\item $R_i$ has one pair of vertical sides and the other pair of sides is parallel to the line segment joining $u_{i}$ and $u_{i+1}$. 
\item The vertical sides have midpoints $u_{i}$ and $u_{i+1}$. 
\item The height of the vertical sides is $n^{1/2-3\e/5}$. 
\end{enumerate} 
It is straightforward to check from Definition \ref{d:reg} that all the rectangles $R_{i}$ are contained in the rectangle $R'$. Now it follows from Theorem \ref{t:tftail1} that, with probability at least $1-e^{-n^{c}}$, each geodesic $\gamma_i$ is contained in the parallelogram $R_i$. In particular, on this event, 
$$ |\gamma_{S}|\leq L(u_0,u_{k}; R').$$
The result now follows using Theorem \ref{t:constrained} and the assumed lower bound on $|u_0-u_{k}|$.
\end{proof}

\begin{figure}[h] 
\centering
\begin{tabular}{ccc}
\includegraphics[width=0.4\textwidth]{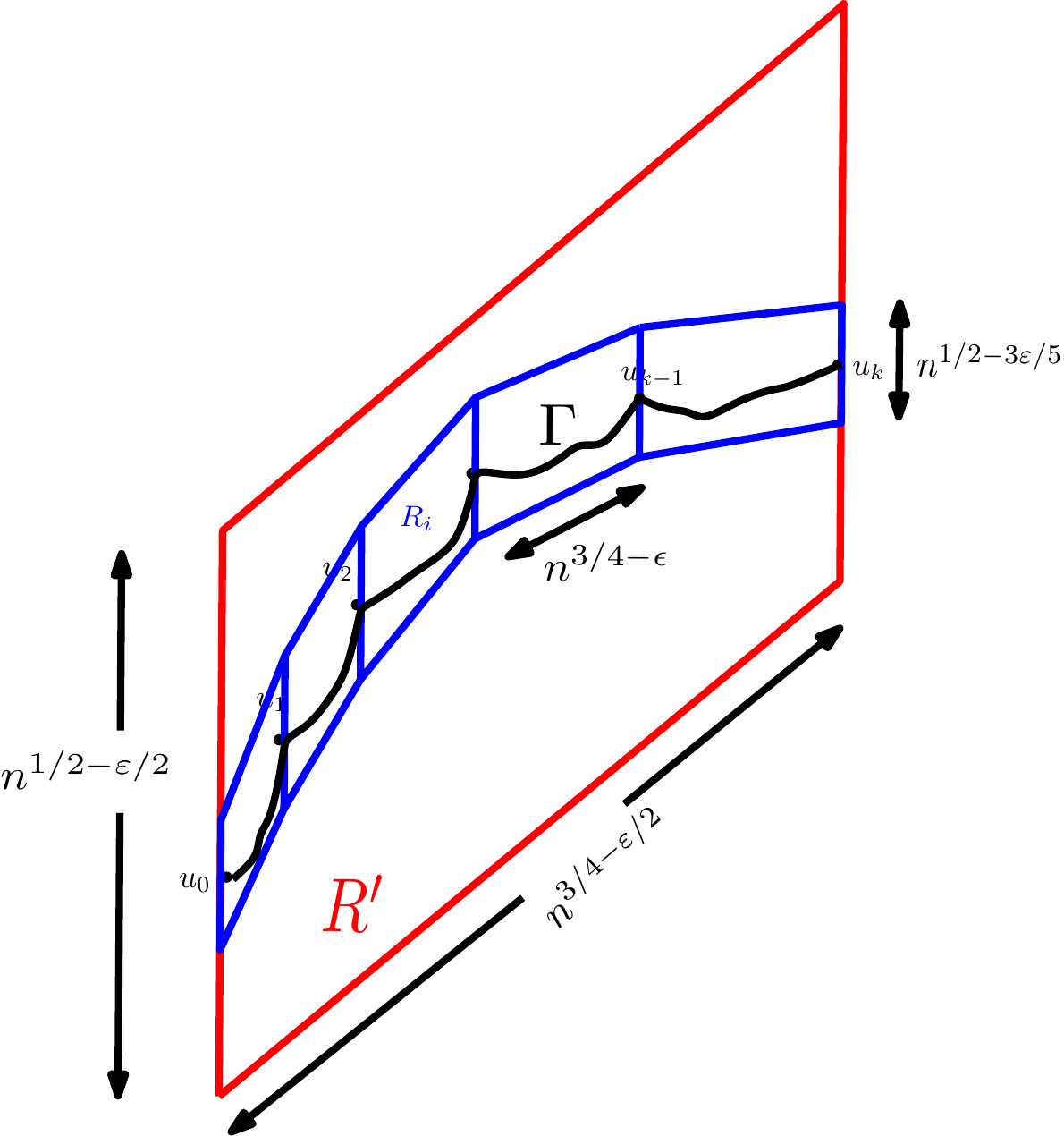} &\includegraphics[width=0.4\textwidth]{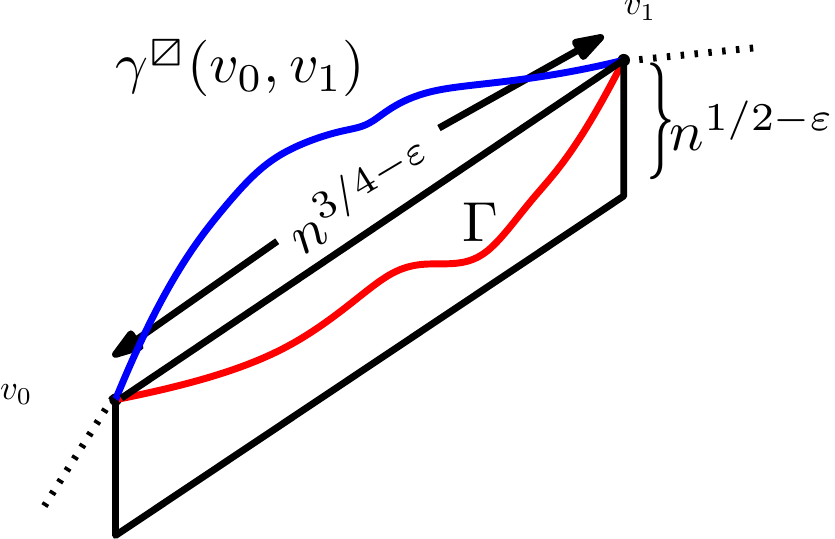} \\
(a) & (b) 
\end{tabular}

\caption{ (a) Illustrates the situation when all the facets are small. Using a priori bounds on the transversal fluctuations of the unconstrained geodesic, it follows that the segment of $\Gamma_{\alpha,n}$ between $u_0$ and $u_k$ lies in a thin parallelogram with width much smaller than the characteristic fluctuation scale. This then implies that the length of this segment is atypically low (see Theorem \ref{t:constrained}).
(b) Illustrates how in this situation one can replace the segment of $\Gamma_{\alpha,n}$ between $u_0$ and $u_k$ by a path that lies strictly above the parallelogram and fluctuates on a characteristic scale, and hence captures more area along with more length, which contradicts extremality of $\Gamma_{\alpha,n}$. This is used in the proof of Theorem \ref{t:mlrub} } 
\label{f:alternatepath20}
\end{figure}
Next we move onto Step 2 where we show the existence of a ``good" path between $u_0$ and $u_{k}$. We have the following proposition. 

\begin{proposition}
\label{p:step2}
Let $S=\{u_0,u_1,\ldots, u_k\}$ be a $(C,\kappa)$-regular sequence. Let $\mathbb{L}_{S}$ denote the union of line segments joining $u_{i}$ to $u_{i+1}$.Then there exists $c>0$ such that, with probability at least $1-e^{-n^{c}}$, there exists a path $\gamma_{*}$
from $u_0$ to $u_{k}$ lying above $\mathbb{L}_{S}$ such that
$$|\gamma_{*}| > \E L(u_0,u_k)-n^{1/4-\e/12}.$$
\end{proposition} 

\begin{proof}
We will use the notation from Lemma \ref{l:reggeo}. Let $u_*$ be the point in $R$ that satisfies the conclusion of that lemma.
Let $\gamma_{1}$ denote a path from $u_0$ to $u_{*}$ lying above the line segment joining $u_0$ and $u_{*}$ that achieves the length $L^{\boxslash}(u_0,u_{*})$. Similarly, let $\gamma_{2}$ denote a path from $u_{*}$ to $u_{1}$ lying above the line segment joining $u_*$ and $u_{1}$ that achieves the length $L^{\boxslash}(u_*,u_{1})$. See Figure \ref{f:corner} (b).

Let $\gamma_{*}$ denote the concatenation of the paths $\gamma_1$ and $\gamma_2$. Clearly $\gamma_{*}$ lies above $\mathbb{L}_{S}$. We shall show that, with overwhelming probability, $\gamma_*$ also satisfies the other condition in the statement of the proposition. It follows by definition and some elementary geometry that both of the equal sides of the isosceles triangle $T$ have gradient in $(\frac{1}{2\kappa}, 2\kappa)$, and also $|u_0-u_*|=|u_*-u_k|=\Theta (n^{3/4-\e/2})$. It follows using Theorem \ref{baikrains} that with probability at least $1-e^{-n^{c}}$ one has 
$$ L^{\boxslash}(u_0,u_*)\geq \E L(u_0,u_*)- n^{1/4-\e/8}; \qquad L^{\boxslash}(u_*,u_k)\geq \E L(u_*,u_k)- n^{1/4-\e/8}. $$
Also observe that $u_*$ is equidistant from $u_0$ and $u_k$ and that this point lies within the parallelogram $R$ whose height is $o(|u_0-u_k|^{2/3})$. Thus, by a standard calculation, 
$$ \E L(u_0,u_*)+ \E L(u_*,u_k) -\E L(u_0,u_k)= o(n^{1/4-\e/6}).$$
Combining these inferences, we find that, with probability at least  $1-e^{-n^{c}}$,  
$$|\gamma_{*}| > \E L(u_0,u_k)-n^{1/4-\e/8}.$$
This completes the proof of the proposition.
\end{proof}

In the final step, we establish that, on $\mathcal{A}_{\e}$, that is if ${\rm MFL}(\Gamma_{n})$ is smaller than $n^{3/4-\e}$, it is extremely likely that there exists a regular sequence whose points are consecutive corners of  $\Gamma^*_{\alpha, n}$ (the least concave majorant of $\Gamma_{\alpha, n}$). 
Let ${\rm Reg}(C,\kappa)$ denote the event that there exist consecutive corners $u_0,u_1,\ldots, u_{k}$ of $\Gamma^*_{\alpha, n}$ such that $\{u_0,u_1,\ldots,u_k\}$ is a $(C,\kappa)$-regular sequence. 

\begin{proposition}
\label{l:hullreg} There exist $C,\kappa$ such that, except for a sub-event of exponentially small probability, 
the event $\mathcal{A}_{\e}$ is contained in the event ${\rm Reg}(C,\kappa)$ i.e. $$\mathbb{P}(\mathcal{A}_{\e}\setminus  {\rm Reg}(C,\kappa)) \le e^{-cn},$$ for some universal constant $c>0$.
\end{proposition}

On $\mathcal{A}_{\e}$, we shall find a regular sequence among interior facets.  
Before proceeding we state an easy geometric lemma  that will be used. The proof is provided in Section \ref{es}.
\begin{figure}[hbt]
\centering
\includegraphics[scale=.4]{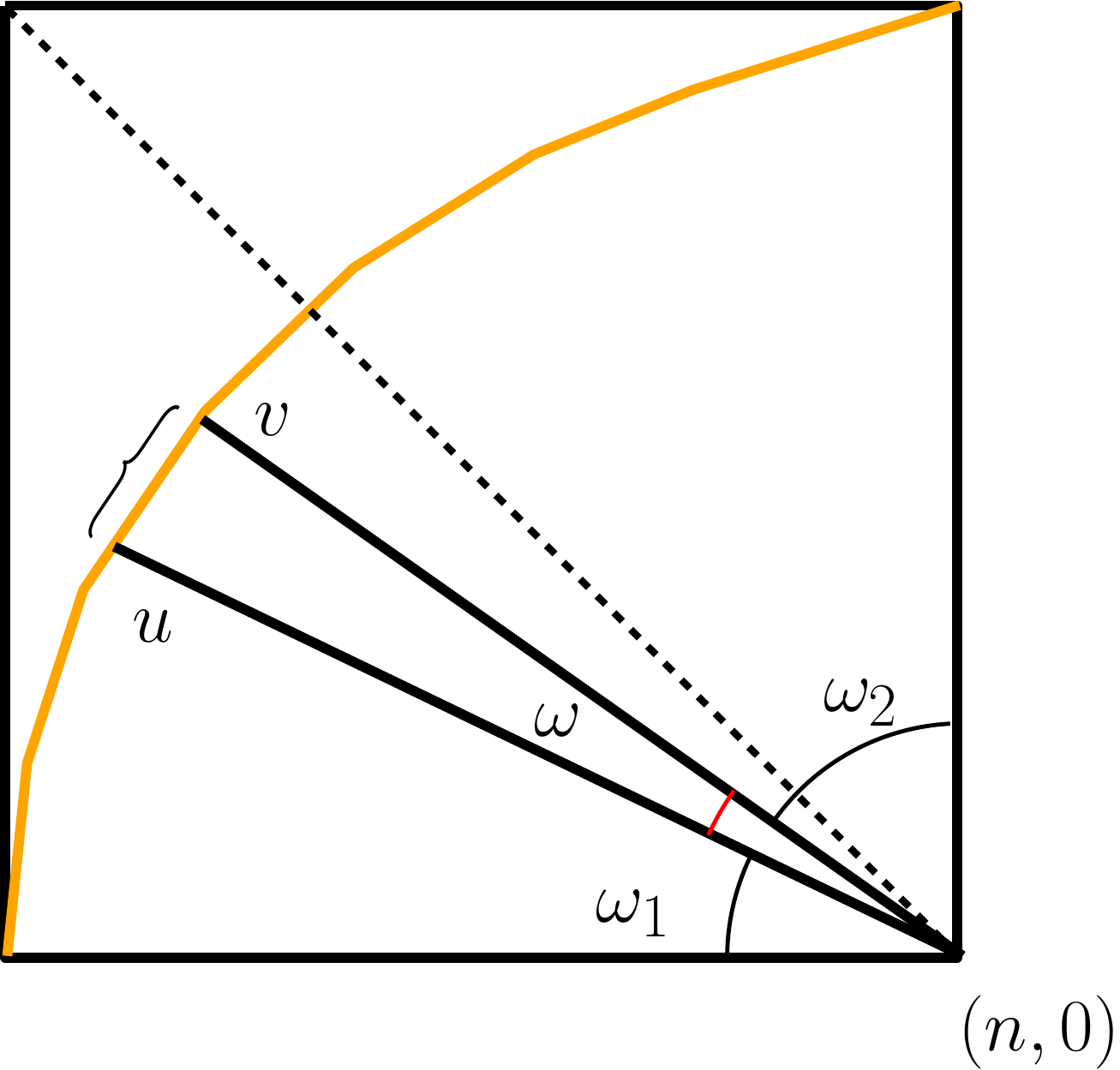}
\caption{Crude estimates on length of facets which subtend an angle $\theta$ at the origin. }
\label{fig1}
\end{figure}

For any two points $p_1,p_2 \in [0,n]^2$, denote by $\theta(p_1,p_2)$ the acute angle between the lines joining the points $p_1$ and $p_2$ to the point $(n,0)$, i.e., the bottom right corner of the square (see Figure \ref{fig1}).
\begin{lemma}\label{crude1}  There exists $\theta_0$ such that with probability at least $1-e^{-cn}$  simultaneously for all $u,v$ on $\Gamma^*_{\alpha,n}$ with $\theta(u,v)\le \theta_0$, then  $|u-v|= \Theta (\theta(u,v) n)$. 
\end{lemma} 

\begin{proof}[Proof of Proposition \ref{l:hullreg}]
Observe that by definition, on $\mathcal{A}_{\e}$, any sequence $u_0\preceq u_1 \preceq \cdots \preceq u_k$ of consecutive corners of $\Gamma^*_{\alpha, n}$ satisfies condition (i) and (iv) of Definition \ref{d:reg}. Fix $\delta>0$.  By Theorem \ref{flat1}, with failure probability at most $e^{-cn}$ all the $\delta-$interior facets have gradient $s$ in the interval $(1/\kappa, \kappa)$ for some $\kappa(\delta)>1$.  For the remainder of the proof, we shall assume that this event occurs. 
Now we  partition all the $\delta$-interior vertices of $\Gamma^*_{\alpha, n}$ (it is easy to see they are non-empty on $\mathcal{A}_{\e}$) into consecutive segments 
\begin{equation}\label{partition1}
(v_{1},v_{2},\ldots, v_{i_1}), (v_{i_1+1},v_{i_1+2}, \ldots, v_{i_2}),\ldots,
\end{equation} (where $v_i$'s are the consecutive corners on $\Gamma^*_{\alpha,n}$)
 such that 
\begin{align*}
\theta(v_{i_j},v_{i_{j+1}-1}) &< n^{-1/4-\e/2} \le \theta(v_{i_j+1},v_{i_{j+1}})
\end{align*}
for all $j=0,1,\ldots,$ where $i_0=1$.

We denote the segment $(v_{i_j},v_{i_j+1}, \ldots, v_{i_{j+1}})$ by $S_j$ 
Thus if $S_1,S_2,\ldots, S_{m}$ are all such segments, then $m=\Theta(n^{1/4+\e/2})$. It follows from convexity and  the boundedness of the gradients (see Lemma \ref{crude1} and Figure \ref{fig1}) that  $|v_{i_j}-v_{i_{j+1}}|=\Theta (n^{3/4-\varepsilon/2})$. Let $\omega_{k}$ denote the angle made by the line segment $(v_{k}, v_{k+1})$ with the positive $x$-axis. For the sequence $S_j$, let $\theta_{j}=\omega_{i_j}-\omega_{i_{j+1}-1}$. Notice that, by convexity, $\theta_j$ is nonnegative for all $j$. Also observe that 
$$\sum_{j}  \theta_{j}\leq \frac{\pi}{2}.$$
This together with Markov's inequality and the fact that the number of $S_j$'s is $\Theta(n^{1/4+\e/2})$ implies that, for some constant $C$, there exists at least one $S_j$ such that $\theta_j \leq Cn^{-1/4-\e/2}$. 
Setting $S=(u_0,u_1,\ldots , u_{k})$ to be equal to such a sequence, we see that $S$ satisfies all conditions in the Definition \ref{d:reg} for some $C,\kappa>1$. This completes the proof of the proposition.
\end{proof}

With all this preparation, finally we are ready to prove Theorem \ref{t:mfllb}.

\begin{proof}[Proof of Theorem \ref{t:mfllb}] 
Let $\alpha\in (0,\frac{1}{2})$ and $\e>0$ be fixed as before. Let $(C,\kappa)$ be such that the conclusion of Proposition \ref{l:hullreg} holds. Let $\mathcal{S}$ denote the event that there exists a $(C,\kappa)$-regular sequence for which at least one of the conclusions in Proposition \ref{p:step1} and Proposition \ref{p:step2} does not hold.
Using the proofs of the said propositions, along with  taking an union bound over all pairs of grid points obtained by coarse graining $[0,n]^2$ (exactly as Corollary \ref{uniform1} followed from Theorem \ref{t:moddev}), that $\P(\mathcal{S})\leq e^{-n^{c}}$ for some $c>0$.
By Proposition \ref{l:hullreg},  it suffices to show that  
$$\mathcal{A}_{\e}\cap \mathcal{S}^c \cap {\rm Reg}(C,\kappa)=\emptyset$$
By way of contradiction, let us assume that the set  is non-empty. Fix a configuration in the set $\mathcal{A}_{\e}\cap \mathcal{S}^c \cap {\rm Reg}(C,\kappa)$, and consider the least concave majorant $\Gamma^*_{\alpha, n}$ of the constrained geodesic. On $\mathcal{A}_{\e}\cap  {\rm Reg}(C,\kappa)$, there exists a $(C,\kappa)$-regular sequence $S=\{u_0,u_1,\ldots , u_{k}\}$ consisting of consecutive corners of $\Gamma^{*}_{\alpha,n}$. On $\mathcal{S}^c$, there exists a path $\gamma_*$ from $u_0$ to $u_{k}$ satisfying the conclusion of Proposition \ref{p:step2}. Let $\gamma_{0}$ denote the restriction of $\Gamma$ between $u_0$ and $u_{k}$. Clearly, 
$$ |\gamma_{0}|\leq \sum_{i} L(u_i,u_{i+1}).$$ It follows from Propositions \ref{p:step1} and \ref{p:step2} that, on $\mathcal{S}^c$,  $|\gamma_{0}|< |\gamma_{*}|$. Also observe that $\gamma_{0}$ lies below the line segments joining $u_{i}$ to $u_{i+1}$ whereas $\gamma_{*}$ lies above these segments. Hence the increasing path $\Gamma'$ from $(0,0)$ to $(n,n)$ that coincides with $\Gamma$ between $(0,0)$ and $u_0$ and between $u_{k}$ and $(n,n)$ and coincides with $\gamma_{*}$ between $u_0$ and $u_{k}$ traps at least as much area as does $\Gamma$. However, by the above discussion, $|\Gamma'|> |\Gamma|$. This contradicts the extremality of $\Gamma$ and completes the proof of the theorem. 
\end{proof}

\subsection{Proof of lower bound of ${\rm MLR}$}
In this subsection we prove Theorem \ref{t:mlrlb}. That is, we show that the maximum local roughness of the facets has scaling exponent at least $\frac{1}{2}$. The basic idea of the proof is to  use KPZ scaling to argue that, if there is a facet of length roughly $n^{3/4}$, then along that facet, the local roughness is likely to be at least $(n^{3/4})^{2/3}=n^{1/2}$. We now make this more precise.

Fix $\varepsilon>0$. We have just proved that, with high probability, there exists a facet of length at least $n^{3/4-\varepsilon}$. In fact, by the proof of Theorem \ref{t:mfllb}, more is true. Fix $\kappa\gg 1$ a large constant. Let $\mathcal{B}_{\varepsilon}=\mathcal{B}_{n,\e,\kappa}$ denote the event that $\Gamma^{*}_{\alpha,n}$ (see Definition \ref{d:mlrmfl}) has a facet of length at least $n^{3/4-\varepsilon}$ with endpoints $v_0$ and $v_{1}$ such that the straight line joining $v_0$ and $v_1$ has gradient  $\in (1/\kappa, \kappa)$. The following lemma is a consequence of the proof of Theorem \ref{t:mfllb}. 

\begin{lemma}
\label{l:be}
For all large enough $\kappa$ (depending on $\alpha$),
$\P(\mathcal{B}_{\varepsilon})\ge 1-e^{-n^c}$ for all $n$ sufficiently large, 
where $c$ is a constant depending only on $\e$ and $\alpha$.
\end{lemma}

\begin{proof}
Note that the proof of Theorem \ref{t:mfllb} in fact showed that there must be a $\delta$-interior facet of length at least $n^{3/4-\varepsilon}$ with sufficiently large probability. Since all the $\delta$-interior facets satisfy the gradient requirement (see Theorem \ref{flat1}), we are done. 
\end{proof}

\begin{proof}[Proof of Theorem \ref{t:mlrlb}]
Let $\mathcal{C}_{\varepsilon}$ denote the event that ${\rm MLR}(\Gamma_n) \leq n^{1/2-\varepsilon}$. Thus, to prove Theorem~\ref{t:mlrlb}, it suffices to show that  $$\P(\mathcal{C}_{\varepsilon}\cap \mathcal{B}_{\varepsilon})\le e^{-n^c}.$$ 

Let $\Gamma$ denote the constrained geodesic. On $\mathcal{C}_{\varepsilon}\cap \mathcal{B}_{\varepsilon}$, let $v_0$ and $v_1$ be corners of $\Gamma^*_{\alpha, n}$ satisfying the conditions in the definition of $\mathcal{B}_{\varepsilon}$. Consider the restriction $\Gamma(v_0,v_1)$ of $\Gamma$ between $v_0$ and $v_1$. As ${\rm MLR}(\Gamma) \leq n^{1/2-\varepsilon}$; it follows that there exists $C=C(\kappa)>0$ such that the restriction of $\Gamma$ (call it $\gamma$) between $v_0$ and $v_1$ is contained in the parallelogram $U$ that has a pair of vertical sides of height $Cn^{1/2-\varepsilon}$ and whose top side is the line segment joining $v_0$ and $v_1$. Now observe that the height of the parallelogram is much smaller than the on scale transversal fluctuation of geodesics between $v_0$ and $v_1$, these fluctuations being at least $$(n^{3/4-\e})^{2/3-\e/100}\ge n^{1/2-{3\e}/{4}} . $$ 
Thus, Corollary \ref{uniform0} implies that the bound 
$$\E(L(v_0, v_1))-|\gamma| \ge |v_0-v_1|^{1/3- 5\varepsilon/9}$$
fails with  probability at most $e^{-n^{c}}$. By Corollary \ref{uniform2}, 
$$
 L^{\boxslash}(v_0,v_1)-\E(L(v_0,v_1)) \ge  -|v_0-v_1|^{\frac{1}{3}+\e/100}
 $$
except, again, on a set of probability at most $e^{-n^{c}}$.
Consider now the path $\Gamma'$ that agrees with $\Gamma$ outside the segment $(v_0,v_1)$ and that between $v_0$ and $v_1$ equals $\gamma^{\boxslash}(v_0,v_1)$, which recall is a path from $v_0$ to $v_1$ lying above the line segment that joins these two points and achieves length $L^{\boxslash}(v_0,v_1)$.
We see then that $\Gamma'$  has greater length, and traps as much area as,~$\Gamma$. 
This contradicts the assumption of length maximality for $\Gamma$; (similarly to the proof of Theorem \ref{t:mfllb}. See Figure \ref{f:alternatepath20}(b).)
\end{proof}

\section{Upper Bound for Scaling Exponents}
\label{s:ub}
This section is devoted to the proof of Theorem \ref{t:mflub}. That is,  it will be shown here that, for a dense set of $\alpha$ that may depend on $n$, the maximum facet length of the constrained geodesic is very likely to be at most 
$n^{3/4+o(1)}$. Theorem \ref{t:mlrub} will then follow by the KPZ scaling enjoyed by transversal fluctuations.  

%
%

We explain the main ideas before giving proofs. Recall Definition \ref{good10}. The basic idea is to show, using Corollary \ref{uniform2}, 
that if a certain value of $\alpha$ is not `good' (recall Definition \ref{good10}),
we may  perform  a ``landgrab" operation. More formally let us assume that there is a facet $(v_0,v_1)$ which has length $n^{3/4+\e}$ and has a reasonable gradient bounded away from zero and infinity. Theorem \ref{baikrains} allows us to find a path $\gamma$ from $v_0$ to $v_1$ which lies above the facet $(v_0,v_1)$ at a characteristic height of $(n^{3/4+\e})^{2/3}$ and has characteristic fluctuations (we have room to allow some error) i.e. by Corollary \ref{uniform2} we can assume 
$$|{\gamma}|\ge \E(L(v_0,v_1)) -(n^{3/4+\e})^{1/3+\e/1000}\ge \E(L(v_0,v_1))-n^{1/4+2\e/5}.$$
Similarly by an Corollary \ref{uniform1} we can assume that $$|\Gamma(v_0,v_1)| \le \E(L(v_0,v_1))+ n^{1/4+2\e/5}$$ 
where $\Gamma(v_0,v_1)$ is the restriction of the constrained geodesic $\Gamma$ between  $v_0$ and $v_1$. 
Thus the path $\Gamma'$ which is obtained from $\Gamma$ by replacing $\Gamma(v_0,v_1)$ by $\gamma$ traps at least  $n^{5/4+{5\e}/{3}}$ more area than  does $\Gamma$ and loses at most $n^{1/4+2\e/5}$ in length. Repeating this operation about $n^{3/4-5\e/3}$ times (we can do that provided there are no good $\alpha$ in this interval), one gains an area that is $\Theta(1)$ while losing at most $O(n^{1-\e})$ in length with very high probability. This contradicts the fact that $\E L_{\alpha_1}(n) -\E L_{\alpha_2}(n)=\Theta(n)$ for $\alpha_1<\alpha_2$ (an easy consequence of Theorem \ref{t:lln} using ${\rm w}_{\alpha}$ is strictly decreasing) using Theorem \ref{concentration1}. Next we provide  the details needed to make the above argument precise. 

Fix $0<\alpha_1 <\alpha_2 <\frac{1}{2}$ and also $\e>0$ and $\delta\in (0,\pi/4)$. Let $\cg_{\alpha_1,\alpha_2}$ denote the event that no $\alpha\in [\alpha_1,\alpha_2]$ is  $(n,\varepsilon, \delta)$-good. Let $B_{n}=B_{n,\e,\kappa}$ be the set of all pairs $(x,y)\in [0,n]^2$ satisfying the following conditions. 
\begin{enumerate}
\item $x\preceq y$.
\item The straight line joining $x$ and $y$ has gradient $\in (\frac{1}{\kappa}, \kappa)$.
\item $|x-y|\geq n^{3/4}$. 
\end{enumerate} 
For $(x,y)\in B_n$, let $A_{x,y}$ denote the event that $|L^{\boxslash}(x,y)-\E L(x,y)|\leq |x-y|^{1/3+\varepsilon/1000}$ and $|L^{*}(x,y)-\E L(x,y)|\leq |x-y|^{1/3+\varepsilon/1000}$ where $L^*(x,y)$ denotes the best path between $x$ and $y$ that is constrained to stay below the straight line joining $x$ and $y$. Let $\mathcal{A}_{\e}$ denote the event that $A_{x,y}$ holds for all $(x,y)\in B_n$. Further let $S_{\delta}$ denote the event in the statement of Theorem \ref{flat1}, i.e., all the $\delta$ interior facets of  $\Gamma^*_{\alpha, n}$ has moderate gradients for all $\alpha\in [\alpha_1, \alpha_2]$. We have the following proposition.

\begin{figure}[h]
\center
\includegraphics[width=0.3\textwidth]{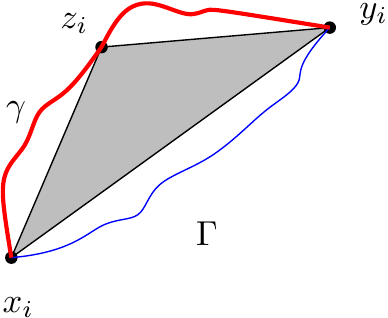}
\caption{The landgrab operation. By replacing the path $\Gamma$ by the path $\gamma$ one gains at least the amount of the area in the shaded region, whereas the event $A_{\varepsilon'}$ ensures that the length loss is not too much.}
\label{f:landgrab}
\end{figure}

\begin{proposition}
\label{p:landgrab}
There exists $\kappa=\kappa(\delta,\e)$ sufficiently small such that, on $\mathcal{A}_{\varepsilon}\cap S_{\delta}\cap \cg_{\alpha_1,\alpha_2}$, there exists $\alpha_* \geq \alpha_2$ for which 
$$L_{\alpha_1}(n)-L_{\alpha_*}(n)\leq (\alpha_{*}-\alpha_1) n^{1-6\varepsilon/5}.$$
\end{proposition}

\begin{proof}
We perform the following recursive construction. Let $\beta_0=\alpha_1$. For $\beta_{i}\in [\alpha_1,\alpha_2]$ construct $\beta_{i+1}$ recursively as follows: find the longest $\delta$-interior facet in $\Gamma_{\beta_i,n}$. On $\cg_{\alpha_1,\alpha_2}$, the longest $\delta$-interior facet $(x_i,y_i)$ has length at least $n^{3/4+\varepsilon}$.  Note that on $S_{\delta}$, the endpoints $x_{i}$ and $y_{i}$ of this facet satisfy  $(x_{i},y_{i})\in B_n$ for some $\kappa=\kappa(\delta)$. Now pick a point $z_{i}$  at orthogonal distance $|x_{i}-y_{i}|^{2/3}$ from  the midpoint of the  line segment joining $x_i$ and $y_{i}$; (see Figure \ref{f:landgrab}). Thus the area of the triangle $(x_i,y_i,z_i)$ is at least $c|x_{i}-y_{i}|^{5/3}$ for some constant $c$ that does not depend on $i$. Set $\beta_{i+1}=\beta_{i}+ cn^{-2}|x_{i}-y_{i}|^{5/3}$. Now consider the path $\gamma$ that coincides with $\Gamma_{\beta_{i},n}$ outside the facet $(x_i,y_i)$, and is formed by concatenating $\gamma^{\boxslash}(x_i,z_i)$ and $\gamma^{\boxslash}(z_i,y_i)$ between $x_i$ and $y_{i}$. Clearly the area under the curve $\gamma$ is at least $(\frac{1}{2}+\beta_{i+1})n^2$. Also, on $A_{\varepsilon}$, 
$$|\Gamma_{\beta_{i},n}|-|\gamma|\leq |x_{i}-y_{i}|^{1/3+\varepsilon/500}.$$
It follows that  
$$L_{\beta_i}(n)-L_{\beta_{i+1}}(n) \leq |x_{i}-y_{i}|^{1/3+\varepsilon/500}.$$

Denote $\alpha_{*}=\beta_{i_0}$, where $i_0$ is the smallest index $i$ for which $\beta_{i}\geq \alpha_2$. It follows that 
$$L_{\alpha_1}(n)-L_{\alpha_*}(n)\leq \sum_{i=0}^{i_0-1} |x_{i}-y_{i}|^{1/3+\varepsilon/500}\leq c^{-1}(\alpha_*-\alpha_1)n^{2} \left(\max_{i\leq i_0} |x_i-y_i|^{-4/3+\varepsilon/500}\right).$$
Since $|x_i-y_{i}|\geq n^{3/4+\e}$ for all $i$, it follows that the final term is at most $n^{-1-4\varepsilon/3}$   and that completes the proof of the proposition.
\end{proof}

We may now  complete the proof of Theorem \ref{t:mflub}. 

\begin{proof}[Proof of Theorem \ref{t:mflub}]
Let $\kappa$ be  large enough that the conclusion of Proposition \ref{p:landgrab} holds. Using Theorem \ref{flat1} and Corollary \ref{uniform}, it follows that, for this choice of $\kappa$ (and for given  $\alpha_1,\alpha_2, \delta$ and $\e$), one has that, for some constant $c>0$,
$$\P[\mathcal{A}_{\e}^{c}\cup S_{\delta}^{c}] \leq e^{-n^c}.$$
Hence, using Proposition \ref{p:landgrab}, it suffices to show that 
\begin{equation}
\label{e:suffice}
\P\left (L_{\alpha_1}(n)-L_{\alpha_2}(n)\leq (1/2-\alpha_1)n^{1-6\e/5}\right )\leq e^{-n^c}
\end{equation}
for some $c>0$. Notice that in the above step we have used the trivial inequality $L_{\alpha_2}(n)\geq L_{\alpha_*}(n)$ for all $\alpha_{*}\in [\alpha_2,\frac{1}{2})$.
It follows now from Theorem \ref{t:lln} that $\E L_{\alpha_1}(n)-\E L_{\alpha_2}(n)\geq \frac{{\rm w}_{\alpha_1}-{\rm w}_{\alpha_2}}{2}n$ for $n$ sufficiently large. The claimed bound \eqref{e:suffice} now follows from Theorem \ref{concentration1} and the fact that ${\rm w}_{\alpha_1}-{\rm w}_{\alpha_2} >0$. This completes the proof.
\end{proof}

\subsection{Proof of Theorem \ref{t:mlrub}}
This derivation is similar to that by which Theorem \ref{t:mlrlb} follows from Theorem \ref{t:mfllb}. 
By Theorem \ref{flat1}, for all $\alpha \in [\alpha_1,\alpha_2]$ and all small enough $\delta$, there is a probability at least $1-e^{-cn}$ that all the $\delta$-interior facets of $\Gamma^*_{\alpha,n}$ are $\kappa-$ steep for some value of $1\le \kappa<\infty$.
Now for any good $\alpha\in [\alpha_1, \alpha_2]$ (note that the set of good $\alpha'$s is a function of the underlying point process and hence is random), 
by definition, all the interior facets of $\Gamma^*_{\alpha,n}$ have length at most~$n^{3/4+\e}$. Moreover, by the uniform bounds  discussed above we may restrict to the case where all the facets are $\kappa$-steep. Suppose that the maximum local roughness restricted to the $\delta$-interior facets  is  at least $n^{1/2+\e}$.  Let $u$ and $v$ be the endpoints of the facet  that attains this maximum local roughness. Let $\gamma(u,v)$ be the segment of $\G_{\alpha,n}$ between $u$ and $v$. 
By using Corollary \ref{uniform0}, we see that, with probability at least $1-e^{-n^c}$, 
$$|\gamma(u,v)|< \E(L(u,v))- S(u,v).$$ On the other hand,
 by Corollary \ref{uniform2} again, we find that, except on an event of probability at most $e^{-n^c}$,
$$L^{\boxslash}(u,v)\ge \E(L(u,v)))- S(u,v).$$
Thus, the path that agrees with $\Gamma$ from $(0,0)$ to $u$ and from $v$ to $(n,n)$ 
and equals the path $\gamma^{\boxslash}(u,v)$ between the points $u$ and $v$ clearly captures at least as much area as does $\G_{\alpha,n}$ and also contains more points  than does $\G_{\alpha,n}$. This  contradicts the extremality of the latter path. 
 \qed

\subsection{Possible extensions and difficulties}\label{areafluc} 
 We end this section with a few remarks related to the first two open problems mentioned in Section \ref{oq}. A natural approach to  quantifying the  proof of Theorem \ref{t:mflub} is via Proposition \ref{p:landgrab}. Namely, one can hope to bound below the density of those $\alpha\in [\alpha_1,\alpha_2]$ that are $(n,\e)$ good (where here we ignore the parameter $\delta$ for the purpose of illustration.) This is because using similar arguments to those used in deriving these two results, one can hope to prove the following uniform bound: with high probability the noise space is such that for all $\alpha\in [\alpha_1,\alpha_2]$ we have $L_{\alpha}(n)-L_{\alpha_*}(n)\le \tilde O(\ell^{1/3})$ for $\alpha_*=\alpha+ \ell^{5/3}n^{-2},$ where $\ell=\max({\rm{MFL}}(\Gamma_{\alpha,n}),n^{3/4}),$ and where the $\tilde O$ notation hides poly-logarithmic multiplicative factors (which arise due to the exceptional nature of the facet endpoints that contributes a polynomial entropy factor).  Thus, the above implies that the length-to-area gradient in scaled coordinates (with length measured in units of $n$ and area measured in units of $n^2$) is $\tilde O(\frac{n}{\ell^{4/3}})=\tilde O(1)$ as by definition $\ell \ge n^{3/4}$. Also, we know that, for all $\alpha$ which are 
  $(n,\e)$ bad, the gradient is polynomially small (since by the land-grab argument (see Figure \ref{f:landgrab}), the length loss is much smaller than the area gain). This coupled with the fact that $L_{\alpha_1}-L_{\alpha_2}=\Theta(n),$ should then allow to conclude that the density of good $\alpha$ is at least $\frac{1}{\tilde O(1)}$.

Notice that our approach for proving upper bounds relies on the joint coupling of the process for various values of $\alpha$. This is quite different in spirit from the known proofs of similar statements in other contexts such as phase separation \cite{AH1} where the upper bound was proven for a fixed value of $\alpha$. The key ingredient there was an understanding of the excess area fluctuation (see Open question (2) in Section \ref{oq}) which then allowed a resampling argument to work. However, since we do not have such bounds, we have to work across various values of $\alpha$ simultaneously. 

 A heuristic argument in our setting proving an upper bound of $n^{5/4+o(1)}$ on the excess area trapped by $\Gamma=\Gamma_{\alpha}$, assuming that $\alpha$ is $(n,\e/10)$ good (recall Definition \ref{good10}) can be made as follows: Suppose that the excess area is more than $n^{5/4+\e}$. Since, by hypothesis the longest facet length is at most $n^{3/4+\e/10}$ , we may find two points $x$ and $y$ on $\Gamma^*_{\alpha,n}$  such that $|x-y|\approx n^{3/4+\e/5}$ and segment of  $\Gamma^*_{\alpha,n}$ between $x$ and $y$ has average curvature (the average distance of the curve from the line segment joining $x$ and $y$ is what it should be on for a circle; see Figure \ref{def} (b)). Now consider shortcutting $\Gamma$ between $x$ and $y$ and hence replacing the subpath of $\Gamma_{\alpha,n}$ between $x$ and $y$ by the unconstrained geodesic between $x$ and $y$. This operation will create an area loss (approximately $n^{5/4+3\e/5}$) which is not enough to violate the area constraint, because the excess area before the shortcut was at least $n^{5/4+\e}$. Moreover, the shortcut increases the overall weight and hence contradicts extremality of $\Gamma_{\alpha,n}$.

\section{Proofs of some of the earlier statements}\label{es}
It remains to provide proofs of Lemmas \ref{l:exist},\ref{l:ac}, \ref{decreasing1}, \ref{crude1} and Theorem \ref{concentration1} that were postponed.

\subsection{Proof of Lemma \ref{l:exist}} This subsection follows closely the arguments presented in \cite[Section 3]{DZ2}. Let $\phi_1,\phi_2,\ldots$ be a sequence of elements in $\mathcal{B}_{\alpha}$ such that $$\lim_{n}J(\phi_n)=J_{\alpha}.$$ 
We equip the space of bounded signed measures on $[0,1]$ with the weak topology generated by $\mathcal{C}=C[0,1]$ (the space of continuous functions on $[0,1]$).  
Since for all $n\ge1,$ $\phi_n$ correspond to sub-probability measures, using compactness, by passing to a subsequence (denoted again by $\{n\}$ for convenience) let $\phi_n$ converge to $\psi\in \mathcal{M},$  in the weak topology.  
This in particular implies $\phi_n(x)$ converges to $\psi(x)$ almost everywhere (since the convergence happens at all continuity points of $\psi$, see e.g., \cite[Section 3.2]{Dur} ). Since $\phi_n$'s are all bounded by $1$ and belong to $\mathcal{B}_{\alpha},$ by the bounded convergence theorem, it follows that $\psi \in \mathcal{B}_{\alpha}$. 
Next we will show that 
\begin{equation}\label{claim203}
\lim_{n}J(\phi_n)\le J(\psi).
\end{equation}
Clearly this implies $J(\psi)=J_{\alpha}$ and hence we would be done.
The claim in \eqref{claim203} follows from the upper semicontinuity of $J(\cdot)$ or equivalently the lower semicontinuity of $-J(\cdot)$.
To show the latter one represents $-J(\cdot)$ as an appropriate Legendre transform. 
To proceed we need some definitions: Clearly \eqref{variational1} extends naturally to all non-negative measures. We extend $J$ to all of $\mathcal{M}$ by setting it to be $\infty$ for all measures which are not non-negative. 
Moreover, for any $f\in \mathcal{C},$ define
 $$\Lambda(f)=\left \{ \begin{array}{cc}
\infty & \text {if } \int_0^1\frac{1}{|f(s)|}ds =\infty  \text{ or if } f\ge 0 \text{ on a set of positive Lebesgue measure}. \\
-\frac{1}{4} \int_{0}^1 \frac{1}{f(s)}ds  & \text{ otherwise}
\end{array}
\right.
$$
Now for any $\phi \in \mathcal{B},$ define $\Lambda^*(\phi)=\displaystyle{\sup_{f\in C[0,1]}  \left[\int_{0}^1 f d(\phi) -\Lambda(f)\right]}$. 
This function, by definition, is lower semicontinuous, since for any $f\in C[0,1],$ and a sequence $\phi_n$ converging to $\phi,$ by definition of weak convergence, $$\lim_{n\to \infty}\int_0^1 f d(\phi_n) = \int_0^1 f d(\phi).$$ Thus the proof is complete by \cite[Lemma 5]{DZ2} which says $\Lambda^*(\phi)=-J(\phi)$. 
\qed

\subsection{Proof of Lemma \ref{l:ac}}
As already hinted right after the statement of Lemma \ref{l:ac}, the proof is by contradiction. Assuming that the singular part of $\psi,$ which we denote by $\psi_s$ (see \eqref{decomposition23}). has positive mass, we create a modification $\psi_1$ such that $J(\psi_1)> J(\psi)$ (see \eqref{variational1}). However in the process, it is possible that we violate the area constraint, i.e. $\int_0^1\psi_1(x){\rm{d}}x < \frac{1}{2}+\alpha$. Thus we make a second modification to construct a function $\psi_2$  with the property $J(\psi_2)> J(\psi),$ and moreover 
it satisfies the area constraint as well. This contradicts the extremality of $\psi$ which was a part of the hypothesis. The details follow.

Let us assume that $\psi_s$ has total mass $k>0$. Then clearly, without loss of generality, we can assume $\psi_s=k\delta_0$, i.e., there is an atom of mass $k$ at zero, since 
\begin{align*}
\int_{0}^1 (\psi_{ac}(x)+k) {\rm{d}}x \ge \int_{0}^1 (\psi_{ac}(x)+\psi_{s}(x)) {\rm{d}}x\ge \frac{1}{2}+\alpha. 
 \end{align*}
Also, because the absolutely continuous part stays the same,  $J(\psi_{ac}+k\delta_0)=J(\psi)$. 
  We will now make some local modifications to contradict extremality of $\psi$. 
\begin{figure}[hbt]
\centering
\includegraphics[scale=.7]{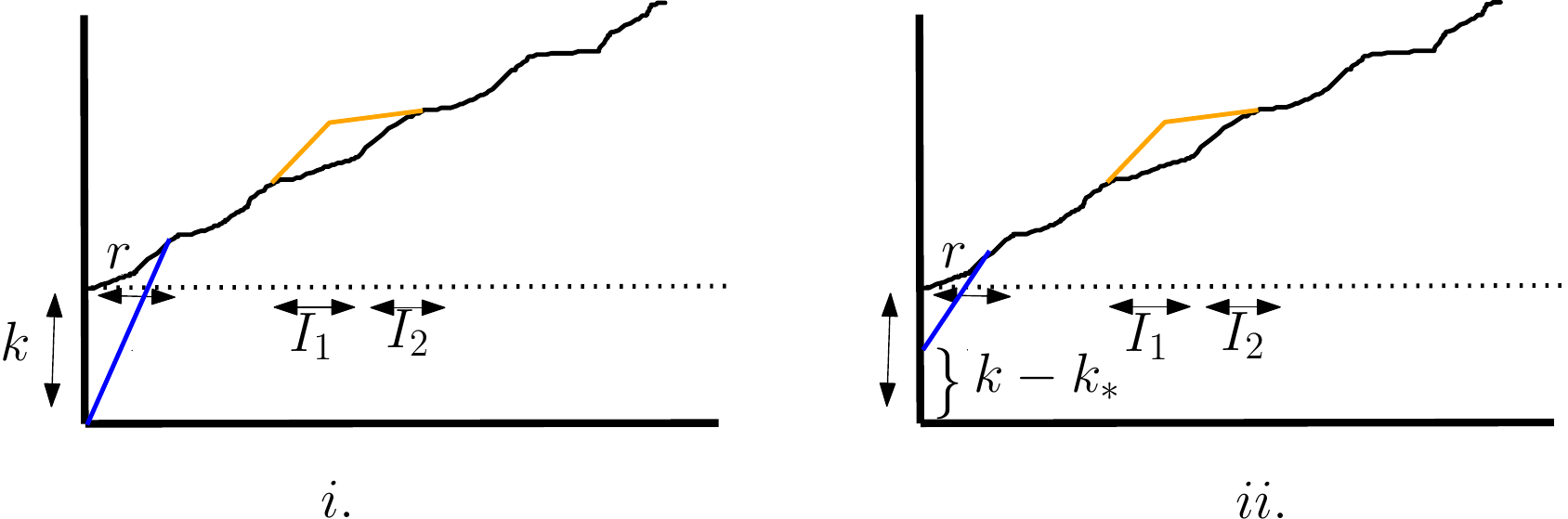}
\caption{Local improvement}
\label{fig9}
\end{figure}

Let $r>0$ be a small number to be chosen later.  Consider a new function $\psi_1$, as illustrated in Figure \ref{fig9}, which is a linear function interpolating $(0,0)$ and $(r,\psi(r))$  and which agrees with $\psi$ on $[r,1]$. Since $\psi(r)=\psi_{ac}(r)+k$,
it follows that
\begin{equation}\label{arealoss}
\int_{0}^1 \psi_1(x){\rm{d}}x \ge \int_{0}^1 \psi(x){\rm{d}}x -\frac{1}{2}kr-O(\psi_{ac}(r)r).
\end{equation}
On the other hand, for all small enough $r$,
\begin{align}\label{lengthgain}
\int_{0}^1 \sqrt{\dot \psi_1(x)} {\rm{d}}x &\ge \int_{0}^1 \sqrt{\dot\psi(x)} {\rm{d}}x  -\int_{0}^r \sqrt{\dot\psi(x)} {\rm{d}}x + \sqrt{kr}\\
\nonumber
&\ge  \int_{0}^1 \sqrt{\dot\psi(x)} {\rm{d}}x -\sqrt{\psi_{ac}(r) r}+ \sqrt{kr},\\
\nonumber
& \ge  \int_{0}^1 \sqrt{\dot\psi(x)} {\rm{d}}x + \frac{\sqrt{kr}}{2}.
\end{align}
The second last inequality uses the bound $\int_0^r \sqrt{\dot \psi(x)} {\rm{d}}x \le \sqrt{r \psi_{ac}(r)},$ which follows from the Cauchy-Schwarz inequality.  To see the last inequality,  note that $\psi_{ac}(r)$ goes to zero as $r$ goes to zero.  
However, it is easy to see that  $$\int_{0}^1  \psi_1(x){\rm{d}}x < \int_0^1 \psi(x){\rm{d}}x .$$ Thus a priori $\psi_{1}$ need not be an element of $\mathcal{B}_{\alpha}$; (see \eqref{restr304}). To ensure that indeed $\int_0^1\psi_{1}(x){\rm{d}}x \ge \frac{1}{2}+\alpha$,  we have to make another modification.
Note that already in the proof of Lemma \ref{l:uni} we argued that a priori even if $\psi$ is not unique, $\dot \psi$ is unique. Thus there exist $c_{\alpha}>0$ and $0<c_1<c_2$  depending only on  $\alpha$ such that the set $(\dot \psi)^{-1}[c_1,c_2]$ has Lebesgue measure at least $c_{\alpha}$. (Note that even though the proof of Lemma \ref{l:uni} relied on Lemma \ref{l:ac}, the uniqueness of the density was argued independently and hence this is not a circular argument.) 

For brevity, let us call the set $(\dot \psi)^{-1}[c_1,c_2]$ by the name~$I$. We now choose arbitrary disjoint subsets of $I\cap [\e,1]$ for some $\e>0$, denoted by $I_1$ and $I_2$, both of measure $k$; (we take the intersection with  $[\e,1]$ to ensure that both the sets are away from $0$).  Note that this can always be done if $\mu(I)> 2k$. In case $\mu(I)\le 2k$, we can modify $\psi_{1}$ by taking it to be  a linear function interpolating $(0,k-k_*)$ and $(r,\psi(r))$ for some $k_*$ small enough so that $\mu(I)>2k_*$; (see Figure \ref{fig9} ii.) 
Nothing changes in any of the arguments that follow as well as the conclusions and hence we will pretend that $k=k_*$ throughout the rest of the argument.
Let us also choose $I_1$ and $I_2$ 
such that $\sup(I_1)< \inf(I_2)$, 
and let $\psi_2(x):=\int_0^x \dot \psi_2(y){\rm{d}}y$ where 
\begin{equation}\label{modi24}
\dot \psi_2:=\dot \psi_1 + \frac{r}{k} \mathbf{1}(I_1) - \frac{r}{k}\mathbf{1}(I_2).
\end{equation}
We will also choose $r<\e$.
Note that since $\dot \psi_1=\dot \psi \ge c_1$ on $I_2$, for all small enough $r,$   $\dot \psi_2$ is non-negative. Also, since the measures of $I_1$ and $I_2$ are the same, $$\int_0^1 \dot \psi_2(x){\rm{d}}x =\int_0^1  \dot \psi_1(x){\rm{d}}x .$$

We now compute the area under the curve $\psi_2(\cdot)$ 
and see how it differs from that of $\psi_1(\cdot)$.
For every $y\in [0,1],$ we have $$\psi_2(y)-\psi_1(y)=\frac{r}{k}[ \mu(I_1\cap [0,y])- \mu(I_2\cap [0,y])].$$
Thus, 
\begin{align}
\label{area123}
\int_0^1 [\psi_2(y)-\psi_1(y)]{\rm{d}}y &=\frac{r}{k} \int_0^1  [\mu(I_1\cap [0,y])-\mu(I_2\cap [0,y])]{\rm{d}}y\\
\nonumber
&= \frac{r}{k} \int_{0}^1 (1-y) [\mathbf{1}(I_1)-\mathbf{1}(I_2)]{\rm{d}}y\\
\nonumber
& \ge rk.
\end{align}
The last inequality follows since, as $\mu(I_1)=\mu(I_2)= k$ and $\sup(I_1)< \inf(I_2)$, the integral above is at least $k^2$.
Thus, the new function $\psi_2(x)$ traps at least as much area as does $\psi(x)$ since 
\begin{align*}
\int_{0}^{1}\psi_2(x){\rm{d}}x -\int_{0}^{1}\psi(x){\rm{d}}x &=[\int_{0}^{1}\psi_2(x){\rm{d}}x -\int_{0}^{1}\psi_1(x){\rm{d}}x ]+[\int_{0}^{1}\psi_1(x){\rm{d}}x -\int_{0}^{1}\psi(x){\rm{d}}x ],\\
&\ge rk-\frac{rk}{2}-O(\psi_{ac}(r)r)>0,
\end{align*}
 for all small enough $r$; (the last inequality uses \eqref{arealoss}).

Moreover,  by Taylor expansion (using \eqref{modi24} and  $\dot \psi_1\ge c_1$ on $I_1 \cup I_2$), 
\begin{align*}
\int_0^1 \sqrt{\dot \psi_2(x)}{\rm{d}}x -\sqrt{\dot \psi_1(x)}{\rm{d}}x  &\ge \int_{I_1} \sqrt{\dot \psi_1(x)}(\frac{r}{2k\dot \psi_1(x)}- C(\frac{r}{k\dot \psi_1(x)})^2){\rm{d}}x  \\
& +\int_{I_2} \sqrt{\dot \psi_1(x)}(-\frac{r}{2k\dot \psi_1(x)}- C(\frac{r}{k\dot \psi_1(x)})^2){\rm{d}}x \\
\nonumber
&\ge -O(\frac{r}{k})k=-O(r).
\end{align*}
We use the bound  $\dot \psi_1> c_1$  and that $\mu(I_1)=\mu(I_2)=k$ crucially in the last inequality, and all the constants  depend only on $c_1$ through the Taylor expansion.
Thus,  by choosing $r\ll k$ we have $\sqrt{rk}\gg r$, and hence using \eqref{lengthgain}, 
\begin{align*}
\int_0^1 \sqrt{\dot \psi_2(x)}{\rm{d}}x -\int_0^1 \sqrt{\dot \psi(x)}{\rm{d}}x &=\int_0^1 \sqrt{\dot \psi_2(x)}{\rm{d}}x -\int_0^1 \sqrt{\dot \psi_1(x)}{\rm{d}}x + \int_0^1 \sqrt{\dot \psi_2(x)}{\rm{d}}x -\int_0^1 \sqrt{\dot \psi_1(x)}{\rm{d}}x ,\\
&\ge -O(r) + \sqrt{kr} -O(\sqrt{\psi_{ac}(r) r}),\\
&\ge \sqrt{kr}/2,
\end{align*}
for all small enough $r$. Hence, we obtain a contradiction to the extremality of $\psi$.
\qed

%

\subsection{Proof of Lemma \ref{decreasing1}}
(i) The proof follows by taking the monotone rearrangement and using uniqueness. Let $\dot \psi_{mon}$  be the monotone rearrangement of $\dot \psi$. 
Notice that, by Fubini's theorem, 
\begin{align}
\label{fubini}
\frac{1}{2}+\alpha\le \int_0^1 \psi(x){\rm{d}}x =\int_0^1 \dot \psi(x) (1-x){\rm{d}}x  \le \int_0^1 (\dot \psi)_{mon}(x) (1-x){\rm{d}}x  = \int_0^1 \psi_1(x){\rm{d}}x  
\end{align} 
where $\dot \psi_1=(\dot \psi)_{mon}$. Only the second inequality above needs justification and it is a consequence of a standard rearrangement inequality. Thus $\psi_1$ satisfies the area constraint. 
Also clearly  $ \int_0^1 \sqrt {\dot\psi}(x){\rm{d}}x = \int_0^1 \sqrt {\dot\psi_1}(x){\rm{d}}x $ as rearrangement keeps integrals unchanged. 
This contradicts the uniqueness of $\psi$ unless $\psi=\psi_{1}$

(ii) Let $\dot\psi$ be zero on a set of positive measure. By the first part of the lemma, this implies that there exists $a$ such that $\psi(y)$ is a constant on the interval $[a,1]$. Note that it is easy to contradict extremality of $\psi$ if $\psi$ is not one on this interval.  In case it is one on this interval, consider the function $\psi_1(x)=1-\psi^{-1}(1-x)$ (where $\psi^{-1}(1)=a$).   We claim that 
\begin{align*}
\int_0^1 \psi_1(y){\rm{d}}y&=\int_0^1 \psi(x){\rm{d}}x ,\,\,\text{ and }
\int_0^1 \sqrt{\dot{\psi_1}(y)}{\rm{d}}y=\int_0^1 \sqrt{\dot\psi(x)}{\rm{d}}x .
\end{align*}
The first equality follows by Fubini's theorem. The second becomes clear after the change of variable $y=1-\phi(x)$ is made.
Now, by uniqueness, $ \psi_1=\psi$, and by Lemma \ref{l:ac}  $\psi$ 
has no singular part. However, $\psi_1$ has an atom of mass $1-a$ at $0$ which implies $a=1$.
\qed

\subsection{Proof of Lemma \ref{crude1}} The reader may find it  useful to refer to Figure~\ref{fig1}. Let $\omega_1$ and $\omega_2$ be the acute angles that the line segments $((n,0),u)$ and $((n,0),v)$ make with the 
$x$-axis and the $y$-axis. Now, without loss of generality, we can assume $0\le \omega_1 \le \pi/4$; otherwise, we could work with $\omega_2$, since clearly one of the two quantities is at most $\pi/4$.
Now, as a simple consequence of Theorem \ref{lln2}, we see that, with probability at least $1-e^{-cn},$ $(n,0)$ is at a $\Theta(n)$ distance from $\Gamma^*_{\alpha,n}$. Also, since $\psi_{\alpha}$  is strictly concave, and $\omega_1 \le \pi/4$, the line segment $(u,v)$  makes at least an angle $c(\alpha)>0$ (depending only on $\alpha$ and not on $u$ and $v$) with the $x$-axis. 
The proof is now completed by  considering the triangle $((n,0),u,v)$   and simple geometric arguments. We omit the details. 
\qed
\subsection{Proof of Theorem \ref{concentration1}} The proof will follow by constructing a suitable martingale and appealing to well known concentration results for martingales. However things are slightly complicated by the possibility that the  martingale may not have bounded increments. A simple truncation argument will allow us to take the increments to be bounded.  
Our proof strategy follows closely arguments in \cite{BB}.
The proof relies on some coarse graining. We start by introducing some notation. 
Let  $$A_{i}:=\{(x,y):i\le x+y < {i+1}\}.$$ 
Also let $B_{i,j}= [{i-1}, {i})\times [{j-1}, {j})$. Fix a number $k$ to  be specified soon.  Recall  the point process $\Pi$ and the path $\Gamma=\Gamma_{\alpha,n}$ from Section \ref{def10}. Let $\Gamma_{k}=\Gamma_{k,\alpha,n}$ be the  increasing path between $(0,0)$ and $(n,n)$ with the same constraints as $\Gamma$ along with the additional  constraint that the intersection with $A_i$ is less than $k$ for all $i \in [0,\ldots, 2n-1]$.  Let us denote  by $|\Gamma_{k}|$ the weight of $\Gamma_{k}$. Consider the Doob Martingale $\{\E\left[|\Gamma_k|\mid \mathcal{F}_i\right]\}_{0\le i\le 2n-1}$, 
along  the filtration $\mathcal{F}_i=\{\Pi \cap A_j : j\le i\}$.   Note that the martingale increments are deterministically bounded by $k$.
The following is a standard consequence of the Azuma-Hoeffding inequality. 
\begin{lemma} $\mathbb{P}(\bigl|  |\Gamma_{k}|-\mathbb{E}(|\Gamma_{k}|)|\ge t)\le e^{-\frac{t^{2}}{4k^2n}}$.
\end{lemma}

In the remainder of the section, we obtain a bound on $\bigl| |\G|-|\G_{k}|\bigr|$ and thus also on $\bigl|\E|\G|-\E|\G_{k}|\bigr|$.
Let $\Pi_*$ be the point process obtained from $\Pi$ by removing points arbitrarily if necessary from  $\Pi \cap B_{i,j}$ to make sure that  $|\Pi_{*}\cap B_{ij}| \le k/3$ for all $1\le i,j\le n$.
Let $\G_{*}:=\G_{\alpha,*}$ be the longest path with the same constraints as $\G_{\alpha}$ in the environment $\Pi_{*}$.

\begin{lemma} Deterministically as a consequence of the definitions it follows that $|\G_{k}| \ge |\G_{*}|$.
\end{lemma}

\begin{proof} Clearly it suffices to show  that $\G_{*} $ intersects none of the $A_{i}$'s at more than $k$ points. The proof is by contradiction: assume that there exists $i$ such that 
$|\G_*\cap A_i|\ge k+1$. Since the points on $\G_*$ are totally ordered (recall the ordering introduced in Section \ref{def1012}), let $(x_1,y_1)$ and $(x_2,y_2)$ be the smallest and largest 
points in the set $\G_* \cap A_i$. By definition of $A_i$, 
\begin{equation}\label{last}
(x_2-x_1)+(y_2-y_1) \le 1,
\end{equation}  
so that $\G_*\cap A_i$ intersects at most three $B_{i,j}'s$,  since the number of different indices it can cross in each direction  is at most one if \eqref{last} is to be satisfied.
Thus $\G_*\cap A_i$ 
has to intersect at least one of the boxes at more than $\frac{k+1}{3}$ points. This contradicts the definition of $\G_*$.
\end{proof}
Thus, it follows that $|\G|-|\G_{k}|\le |\G|-|\G_*| \le C$, where $C=\sum_{i,j} \max(|\Pi_{i,j}|-k/3,0)$ and $\Pi_{i,j}=\Pi \cap B_{i,j}$. Taking $k=6\frac{\log n}{\log\log n}$, the proof is now complete in view of the next result. 
\begin{lemma}The random variable $C$ satisfies the following:
\begin{enumerate} 
\item
$\P(C\ge \lambda n^{1/2}\frac{\log n}{\log \log n}) \le 2 \lambda ^2 e^{-\lambda^2}$
\item $\E(C)\le 1$.
\end{enumerate}
\end{lemma}

This is a simple consequence of the observation that $|\Pi_{i,j}|$ are independent Poisson variables with mean one and the  following tail bound of a standard Poisson variable $X$:
$P(X>r)\le e^{-r\log r +r}$.
\qed

\bibliography{areatrap}

\begin{thebibliography}{10}

\bibitem{AD95}
David Aldous and Persi Diaconis.
\newblock Hammersley's interacting particle process and longest increasing
  subsequences.
\newblock {\em Probab. Th. Rel. Fields}, 103:199--213, 1995.

\bibitem{Alexander}
Kenneth~S Alexander.
\newblock Cube--root boundary fluctuations for droplets in random cluster
  models.
\newblock {\em Communications in Mathematical Physics}, 224(3):733--781, 2001.

\bibitem{FPPsurvey}
Antonio Auffinger, Michael Damron, and Jack Hanson.
\newblock 50 years of first passage percolation.
\newblock {\em arXiv preprint arXiv:1511.03262}, 2015.

\bibitem{BDJ99}
Jinho Baik, Percy Deift, and Kurt Johansson.
\newblock On the distribution of the length of the longest increasing
  subsequence of random permutations.
\newblock {\em J. Amer. Math. Soc}, 12:1119--1178, 1999.

\bibitem{br01}
Jinho Baik and Eric~M. Rains.
\newblock Symmetrized random permutations.
\newblock In {\em Random Matrix Models and Their Applications, volume 40 of
  Mathematical Sciences Research Institute Publications}, pages 1--19, 2001.

\bibitem{br}
Jinho Baik and Eric~M. Rains.
\newblock Symmetrized random permutations.
\newblock MSRI volume: Random Matrix Models and Their Applications(40):1--19,
  2001.

\bibitem{BH17}
Riddhipratim Basu and Alan Hammond.
\newblock Localization of near geodesics in {B}rownian {L}ast {P}assage
  {P}ercolation.
\newblock In preparation.

\bibitem{BSS14}
Riddhipratim Basu, Vladas Sidoravicius, and Allan Sly.
\newblock Last passage percolation with a defect line and the solution of the
  {S}low {B}ond {P}roblem.
\newblock Preprint arXiv 1408.3464.

\bibitem{BLPR}
Marek Biskup, Oren Louidor, Eviatar~B Procaccia, and Ron Rosenthal.
\newblock Isoperimetry in two-dimensional percolation.
\newblock {\em arXiv preprint arXiv:1211.0745}, 2012.

\bibitem{BB}
B{\'e}la Bollob{\'a}s and Graham Brightwell.
\newblock The height of a random partial order: concentration of measure.
\newblock {\em The Annals of Applied Probability}, pages 1009--1018, 1992.

\bibitem{CD13}
Sourav Chatterjee and Partha~S. Dey.
\newblock Central limit theorem for first-passage percolation time across thin
  cylinders.
\newblock {\em Probability Theory and Related Fields}, 156(3):613--663, 2013.

\bibitem{DZ2}
Jean-Dominique Deuschel and Ofer Zeitouni.
\newblock Limiting curves for iid records.
\newblock {\em The Annals of Probability}, pages 852--878, 1995.

\bibitem{DZ1}
Jean-Dominique Deuschel and Ofer Zeitouni.
\newblock On increasing subsequences of iid samples.
\newblock {\em Combinatorics, Probability and Computing}, 8(03):247--263, 1999.

\bibitem{DPM17}
Partha~S. Dey, Ron Peled, and Matthew Joseph.
\newblock Longest increasing path within the critical strip.
\newblock Preprint.

\bibitem{DKS}
Roland~L. Dobrushin, Roman Koteck{\`y}, and S~Shlosman.
\newblock {\em Wulff construction: a global shape from local interaction},
  volume 104.
\newblock American Mathematical Society Providence, Rhode Island, 1992.

\bibitem{Dur}
Rick Durrett.
\newblock {\em Probability: theory and examples}.
\newblock Cambridge university press, 2010.

\bibitem{FS}
Patrik~L Ferrari and Herbert Spohn.
\newblock Constrained brownian motion: fluctuations away from circular and
  parabolic barriers.
\newblock {\em Annals of probability}, pages 1302--1325, 2005.

\bibitem{G16}
Julian Gold.
\newblock Isoperimetry in supercritical bond percolation in dimensions three
  and higher.
\newblock {\em arXiv preprint arXiv:1602.05598}, 2016.

\bibitem{AH3}
Alan Hammond.
\newblock Phase separation in random cluster models iii: Circuit regularity.
\newblock {\em Journal of Statistical Physics}, 142(2):229--276, 2011.

\bibitem{AH1}
Alan Hammond.
\newblock Phase separation in random cluster models i: Uniform upper bounds on
  local deviation.
\newblock {\em Communications in Mathematical Physics}, 310(2):455--509, 2012.

\bibitem{AH2}
Alan Hammond.
\newblock Phase separation in random cluster models ii: The droplet at
  equilibrium, and local deviation lower bounds.
\newblock {\em Ann. Probab.}, 40(3):921--978, 05 2012.

\bibitem{HP}
Alan Hammond and Yuval Peres.
\newblock Fluctuation of a planar brownian loop capturing a large area.
\newblock {\em Transactions of the American Mathematical Society},
  360(12):6197--6230, 2008.

\bibitem{IS}
Dmitry Ioffe and Roberto~H Schonmann.
\newblock Dobrushin--koteck{\`y}--shlosman theorem up to the critical
  temperature.
\newblock {\em Communications in mathematical physics}, 199(1):117--167, 1998.

\bibitem{J00}
Kurt Johansson.
\newblock Transversal fluctuations for increasing subsequences on the plane.
\newblock {\em Probability theory and related fields}, 116(4):445--456, 2000.

\bibitem{LogShep77}
B.F. Logan and L.A. Shepp.
\newblock A variational problem for random young tableaux.
\newblock {\em Advances in Math.}, 26:206--222, 1977.

\bibitem{LM01}
Matthias L{\"o}we and Franz Merkl.
\newblock {Moderate deviations for longest increasing subsequences: The upper
  tail.}
\newblock {\em Comm. Pure Appl. Math.}, 54:1488--1519, 2001.

\bibitem{LMS02}
Matthias L{\"o}we, Franz Merkl, and Silke Rolles.
\newblock {Moderate deviations for longest increasing subsequences: The lower
  tail.}
\newblock {\em J. Theor. Probab.}, 15(4):1031--1047, 2002.

\bibitem{sepLDP}
Timo Sepp{\"a}l{\"a}inen.
\newblock Large deviations for increasing sequences on the plane.
\newblock {\em Probability Theory and Related Fields}, 112(2):221--244, 1998.

\bibitem{VerKer77}
A.M. Vershik and S.V. Kerov.
\newblock Asymptotics of the plancherel measure of the symmetric group and the
  limiting form of young tables.
\newblock {\em Soviet Math. Dokl.}, 18:527--531, 1977.
\newblock Translation of Dokl. Acad. Nauk. SSSR 233 (1977) 1024-1027.

\end{thebibliography}
\bibliographystyle{plain}

\end{document}